\documentclass[12pt,a4paper]{amsart}
\usepackage[totalwidth=15.75cm,totalheight=22.275cm]{geometry}
\usepackage{amsmath,amsfonts,amssymb,verbatim,graphicx,float,amsthm}
\usepackage{color}

\usepackage[shortlabels]{enumitem}
\usepackage{hyperref}

\numberwithin{equation}{section}

\newcommand{\deriv}{{\rm d}}

\newcommand{\interior}{\operatorname{int}}
\newcommand{\direct}{\operatorname{direct}}
\newcommand{\indirect}{\operatorname{indirect}}
\newcommand{\comp}{\operatorname{comp}}
\newcommand{\closure}{\operatorname{cl}}
\newcommand{\erf}{\operatorname{erf}}
\newcommand{\B}{\mathcal{B}}
\newcommand{\AV}{\operatorname{AV}}
\newcommand{\CV}{\operatorname{CV}}
\newcommand{\CP}{\operatorname{C}}

\newcommand*{\defeq}{\mathrel{\vcenter{\baselineskip0.5ex \lineskiplimit0pt
                     \hbox{\scriptsize.}\hbox{\scriptsize.}}}%
                     =}
\newcommand{\eqdef}{=\mathrel{\vcenter{\baselineskip0.5ex \lineskiplimit0pt
                     \hbox{\scriptsize.}\hbox{\scriptsize.}}}}

\makeatletter
\def\swappedhead#1#2#3{%
  \thmnumber{\@upn{\the\thm@headfont #2\@ifnotempty{#1}{.~}}}%
  \thmname{#1}%
  \thmnote{ {\the\thm@notefont(#3)}}}
\makeatother

 \newtheoremstyle{changebreak}
   {9pt}
   {9pt}
   {\itshape}
   {}
   {\bfseries}
   {.}
   {\newline}
   {}

 \newtheoremstyle{defnbreak}
   {9pt}
   {9pt}
   {\normalfont}
   {}
   {\bfseries}
   {.}
   {\newline}
   {}

 \newtheoremstyle{claimstyle}%
   {}
   {}
   {\normalfont}
   {}
   {\itshape}
   {.}
   { }
   {\thmnote{#3}}

\newenvironment{subproof}{\begin{proof}}{%
               \end{proof}}

\theoremstyle{claimstyle}
\newtheorem*{varclaim}{}

\newenvironment{claim}{\begin{varclaim}[Claim]}{\end{varclaim}}
\newenvironment{remark}{\begin{varclaim}[Remark]}{\end{varclaim}}

\swapnumbers

\theoremstyle{changebreak}
\newtheorem{thm}{Theorem}[section]%

\newtheorem{lem}[thm]{Lemma}%
\newtheorem{cor}[thm]{Corollary}%
\newtheorem{corollary}[thm]{Corollary}
\newtheorem{prop}[thm]{Proposition}%
\newtheorem{obs}[thm]{Observation}%

\newtheorem{proposition}[thm]{Proposition}

\newtheorem{theorem}[thm]{Theorem}

\theoremstyle{defnbreak}
\newtheorem{rmk}[thm]{Remark}%
\newtheorem{question}[thm]{Question}%
\newtheorem{defn}[thm]{Definition}%

\newcommand{\N}{\mathbb{N}}
\newcommand{\Z}{\mathbb{Z}}
\newcommand{\R}{\mathbb{R}}
\newcommand{\C}{\ensuremath{\mathbb{C}}}
\newcommand{\Ch}{\hat{\mathbb{C}}}
\newcommand{\D}{\mathbb{D}}
\newcommand{\eps}{\ensuremath{\varepsilon}}

\newcommand{\im}{\operatorname{Im}}
\newcommand{\dist}{\operatorname{dist}}
\newcommand{\diam}{\operatorname{diam}}

\renewcommand{\theta}{\vartheta}
\renewcommand{\phi}{\varphi}

\begin{document}
\title[Connected preimages of simply-connected domains]{On connected preimages of simply-connected domains under entire functions}

\author{Lasse Rempe-Gillen} 
\address{Dept. of Mathematical Sciences \\
	 University of Liverpool \\
   Liverpool L69 7ZL\\
   UK \\ ORCiD: 0000-0001-8032-8580}
\email{l.rempe@liverpool.ac.uk}
\author{Dave Sixsmith}
\address{Dept. of Mathematical Sciences \\
	 University of Liverpool \\
   Liverpool L69 7ZL\\
   UK \\ ORCiD: 0000-0002-3543-6969}
\email{djs@liverpool.ac.uk}
\thanks{The first author was partially supported by a Philip Leverhulme Prize.}

%
%
%
%
\begin{abstract}
  Let $f$ be a transcendental entire function, and let 
     $U,V\subset\C$ be disjoint simply-connected domains. Must one of
    $f^{-1}(U)$ and  $f^{-1}(V)$ be disconnected?

  In  1970, Baker \cite{baker1970} implicitly gave a positive answer to this
     question, in order to prove that a transcendental entire function cannot have two disjoint completely invariant domains.
(A domain $U\subset \C$ 
   is completely invariant under $f$ 
   if $f^{-1}(U)=U$.) 

  It was recently observed by Julien Duval that there is a flaw in Baker's argument
    (which has also been used in  later generalisations and extensions of Baker's result). We show that the answer to the above question is 
    negative; so this flaw cannot be repaired. Indeed, for the function $f(z)= e^z+z$, there is a collection of 
    \emph{infinitely many} pairwise disjoint simply-connected domains, each with connected preimage.
    We also answer a long-standing question of 
  Eremenko
  by giving an example of a transcendental meromorphic function, with infinitely many poles, which has the same property.

  Furthermore, we show that there exists a function  $f$ with the above properties
   such that additionally the set of singular values $S(f)$ is bounded; in other words, $f$ belongs to the 
   \emph{Eremenko-Lyubich class}. On the other hand, if $S(f)$ is finite (or if certain additional hypotheses are imposed),
    many of the original results do hold.
 
For the convenience of the research community, we also include a description of the error
in the proof of \cite{baker1970}, and a summary of other papers that are affected.
\end{abstract}
\maketitle
%
%
%
%
\section{Introduction}

Almost half a century ago, Baker \cite{baker1970} proved that a transcendental entire function cannot have
  two disjoint completely invariant domains; in particular, the Fatou set of such a function has either one or 
  infinitely many connected components. (Since we do not focus on dynamics in this paper, we refer to
  \cite{walteriteration} for background and definitions.)
However, in 2016 Julien Duval observed
that there is a flaw in Baker's proof. It follows that the question of whether a transcendental 
entire function can have two disjoint completely invariant domains remains open.
The same flaw is also found in several subsequent proofs, which
aimed to sharpen or  generalise Baker's original result.

Baker's proof in \cite{baker1970} is topological, rather than dynamical, and (if correct) would 
give a positive answer to the following question. 
\begin{question}[Connected preimages of simply-connected domains]
\label{ques:main}
If $f$ is a transcendental entire function, and $G_1$ and $G_2$ are disjoint simply-connected domains,
then is it true that at least one of $f^{-1}(G_1)$ and $f^{-1}(G_2)$ is disconnected?
\end{question}

The main aim of this article is to give a negative answer to Question~\ref{ques:main}.

\begin{thm}[Connected preimages]
\label{thm:easyexample}
   Let $f(z)=e^z+z$. Then there is an infinite sequence $(U_j)_{j=1}^{\infty}$ of 
     pairwise disjoint simply-connected domains such that $f^{-1}(U_j)$ is connected for all $j$. 
\end{thm}

   Although Theorem~\ref{thm:easyexample} does not answer Baker's original question 
     about completely invariant domains, it shows that his purely topological argument cannot be repaired.
     New ingredients, which involve the dynamics of the function under consideration, 
     would therefore be required to resolve this problem.

   While the function  $f$ in Theorem~\ref{thm:easyexample} is rather simple, the 
    structure of the domains $(U_j)$ is very complicated; they are constructed through a careful 
    recursive procedure that is somewhat reminiscent of the famous ``Lakes of Wada''. We show
    that some similar complexity is necessary, by establishing that
   the answer to Question~\ref{ques:main} is positive 
   when certain additional hypotheses are imposed on the domains
    $G_1$ and $G_2$. 
In the statement of the following theorem, $S(f)$ 
   denotes the closure of the set of all critical and finite asymptotic values of $f$ in $\C$.
Also, $\mathcal{S}$ is the \emph{Speiser class}, consisting of those transcendental entire functions $f$ for which $S(f)$ is finite.
Note that part \ref{walteralexresult} of this theorem is an immediate consequence of \cite[Theorem 1]{walteralex}, 
   which does not depend on Baker's argument, and is included here only for completeness.

\begin{thm}[Disconnected preimages]
\label{thm:itdoeshold}
Suppose that $f$ is a transcendental entire function, and that $G_1, G_2$ are disjoint simply-connected domains
  such that $f^{-1}(G_1)$ is connected. 
  If any of the following conditions hold, then $f^{-1}(G_2)$ is disconnected.
\begin{enumerate}[(a)]
\item $G_1$ is bounded and its closure does not separate $G_2$ from infinity.%
\label{boundeddomain}
\item $G_1 \cap S(f)$ is compactly contained in $G_1$.\label{compactlyc}
\item $f \in \mathcal{S}$.\label{inS}
\item The domain $f^{-1}(G_1)$ contains two asymptotic curves tending to different \emph{tran\-scendental
    singularities} of $f^{-1}$.\label{onlyonesingularity}
\item There exists $\xi \in \partial G_1 \cap \partial G_2$ such that $\xi$ is accessible from both $G_1$ and $G_2$.\label{sharedboundary}
\item Infinity is accessible from $G_1$.\label{infaccessible}
\item $\overline{G_1} \cap \overline{G_2} = \emptyset$.\label{disjointclosure}
\item $f$ has an omitted value.\label{walteralexresult}
\end{enumerate}
\end{thm}
\begin{remark}
 In condition~\ref{onlyonesingularity}, an asymptotic curve is a curve to $\infty$
  along which $f$ converges to an asymptotic value $a\in\Ch$. Two curves
  tend to the same singularity over $a$ if they tend to infinity within the 
  same connected component of $f^{-1}(\Delta)$, for every connected open
  neighbourhood $\Delta$ of~$a$; compare
  \cite{bergweilereremenkosingularities}. 
     In particular,~\ref{onlyonesingularity} holds whenever $G_1$ 
    contains two different asymptotic values,
    or contains one asymptotic value $a$ such that the preimage of any small disc around $a$ has at least two connected
    components that are mapped by $f$ with infinite degree.
\end{remark}

Turning briefly to the dynamics of transcendental entire functions, the following is an easy consequence of Theorem~\ref{thm:itdoeshold}, and is at least a partial result towards the one proved in \cite{baker1970}. 
(We again refer to \cite{walteriteration} for definitions.)
\begin{cor}[Completely invariant domains]
\label{corr:itdoeshold}
Suppose that $f$ is a transcendental entire function, and that $G_1, G_2$ are distinct Fatou components of $f$.
If any of the conditions of Theorem~\ref{thm:itdoeshold} hold, then
   at most one of $G_1$ and $G_2$ is completely invariant.
\end{cor}
Observe that a Siegel disc cannot be  completely invariant, as the map is injective thereon.
  Moreover, infinity is accessible from any Baker domain by definition. Hence, if $f$ has 
  more than one
   completely invariant Fatou component, each such component must be an attracting or parabolic basin.

\subsection*{Other examples}
We also exhibit three other functions having properties similar to those in Theorem~\ref{thm:easyexample}. First we consider the case of meromorphic functions. There are straightforward examples of meromorphic functions with 
two simply-connected domains each with connected preimage. For example, we can take $f(z) = \tan z$, and 
let $G_1, G_2$ be the upper and lower half-plane respectively. In a question closely related to 
Question~\ref{ques:main}, Eremenko \cite{eremenkotalk} asked whether a non-constant meromorphic function 
can have \emph{three} disjoint simply-connected regions each with connected preimage. We show that this is 
indeed possible, even for a meromorphic function with infinitely many poles.

\begin{thm}[Meromorphic functions and connected preimages]
\label{thm:meroexample}
   Let $f(z)=\tan z + z$. Then there is an infinite sequence $(U_j)_{j=1}^{\infty}$ of 
     pairwise disjoint simply-connected domains such that $f^{-1}(U_j)$ is connected for all $j$. 
\end{thm}

\begin{rmk}
Note that another meromorphic example can be obtained directly from Theorem~\ref{thm:easyexample}
    by post-composing the function $z \mapsto e^z + z$ with a fractional linear transformation that takes a point of 
   $\C \setminus (U_1 \cup U_2)$ to infinity.
\end{rmk}

We next ask whether condition \ref{inS} in Theorem~\ref{thm:itdoeshold} can be weakened to require only that $f \in \mathcal{B}$.
Here $\mathcal{B}$ is the \emph{Eremenko-Lyubich class} consisting of those transcendental entire functions for which the set of singular values is bounded.
We show that this is not the case. We also use this example to show that the domains with connected preimages can be bounded.

\begin{thm}[Connected preimages in $\B$]
\label{thm:classbexample}
  There is a transcendental entire function $f\in\B$ such that there is an infinite sequence $(U_j)_{j=1}^{\infty}$ of 
     pairwise disjoint simply-connected domains such that $f^{-1}(U_j)$ is connected for all $j$.
     Moreover, these domains can be taken to be bounded.
\end{thm} 

Our final example addresses the question whether condition \ref{onlyonesingularity} in Theorem~\ref{thm:itdoeshold} can be 
   weakened to require only that $f^{-1}(G_1)$ contains some asymptotic curve. 
   Again, this is not the case.

\begin{thm}[Examples with asymptotic values]
\label{thm:AVexample}
  There exists a transcendental entire function $f$ 
     and pairwise disjoint simply-connected domains $U$ and $V$, each with connected preimage
     and each containing exactly one asymptotic value of $f$.
\end{thm} 

The asymptotic values in our example are in fact \emph{logarithmic asymptotic values} (see Section~\ref{sect:AVexample}). In \cite[Lemma~11]{eremenkoandlyubich}, it is stated that 
   any completely invariant domain of an entire function must contain all logarithmic asymptotic values of $f$. However, the proof uses Baker's flawed argument, and
   would again imply that any simply-connected domain with connected preimage contains all such values. Hence, Theorem~\ref{thm:AVexample} shows that
   this proof also cannot be repaired. Compare the discussion in Section~\ref{sect:appendix}.

\subsection*{Topological results}
Theorems~\ref{thm:easyexample}, \ref{thm:meroexample} and \ref{thm:classbexample} all follow
from our main construction, which is topological in nature. We begin with some preliminary definitions.

 \begin{defn}[Branched coverings]
   A function $f\colon X\to Y$ between oriented topological surfaces is a \emph{branched covering} if every 
    point $w\in Y$ has a simply-connected neighbourhood $D\subset Y$ with the following property.
     If $\tilde{D}\subset X$ is a connected component  of $f^{-1}(D)$, then  $\tilde{D}$ is simply-connected
      and $f\colon \tilde{D}\to D$ maps like $z\mapsto z^d$ (up to orientation-preserving homeomorphisms), for some $d\geq 1$. 

    If $f$ is a branched covering, then we denote by $\CV(f)$ the  set of critical  values of  $f$, and
      by $\CP(f)$ the set of critical points of $f$. 
  \end{defn}

Our construction applies to a certain class of branched coverings, which we define as follows.
 \begin{defn}[Vanilla functions]
   Let $f\colon X\to  Y$ be a branched covering between non-compact simply-connected surfaces $X$ and $Y$. 
     We say that $f$ is \emph{vanilla} if all the following conditions hold.
     \begin{enumerate}[(a)]
       \item All critical points of $f$ are simple; in other words $f$ maps topologically like
       $z \mapsto z^2$ in a neighbourhood of each critical point.  
       \item $\CV(f)$ is a discrete infinite subset of $Y$.
       \item For every $v\in \CV(f)$, $f^{-1}(v)$ contains exactly one element of $\CP(f)$.
     \end{enumerate}
 \end{defn}

We are now able to state our main topological result.
\begin{thm}[Connected preimages for vanilla functions]
\label{thm:maintopological}
  Let $f\colon X\to Y$ be vanilla. Then there exist disjoint simply-connected domains $U,V~\subset~Y$ such that
    $f^{-1}(U)$ and $f^{-1}(V)$ are connected.
\end{thm}

Observe that by part~\ref{compactlyc} of Theorem~\ref{thm:itdoeshold}, the domains $U$ and $V$
   from Theorem~\ref{thm:maintopological} must both contain infinitely many critical values of $f$. 
   It follows easily that the restriction  $f\colon f^{-1}(V)\to V$ is itself again vanilla, and 
   hence we can apply Theorem~\ref{thm:maintopological} again, obtaining two simply-connected
   subdomains of $V$, each with connected preimage. Continuing inductively, we obtain the following
   corollary.
\begin{cor}[Infinitely many domains with connected preimages]
\label{cor:maintopological}
  Let $f\colon X\to Y$ be vanilla. Then there is an infinite sequence $(U_j)_{j=1}^{\infty}$ of 
     pairwise disjoint simply-connected subdomains of $Y$ such that $f^{-1}(U_j)$ is connected for all $j$. 
\end{cor}

\begin{rmk}[Common boundaries]\label{rmk:commonboundary}
By a relatively straightforward modification to the construction in Theorem~\ref{thm:maintopological}, we can 
   ensure additionally that $\partial U = \partial V$; 
    see Observation~\ref{obs:commonboundary}.
With a more complicated modification, we can even ensure that the infinitely many domains of Corollary~\ref{cor:maintopological} have a common boundary. We omit the detail in order to keep our presentation as simple as possible.
\end{rmk}

\subsection*{Structure} 
The structure of this paper is as follows. To help orient the reader, we begin by giving an outline of the proof of Theorem~\ref{thm:maintopological} in Section~\ref{sect:overview}. In Section~\ref{sect:preliminary} we collect preliminary results relating to branched coverings and vanilla maps. 
Section~\ref{sect:construction} is dedicated to the proof of Theorem~\ref{thm:maintopological}. In Section~\ref{sect:examples} we deduce Theorems~\ref{thm:easyexample},~\ref{thm:meroexample} and~\ref{thm:classbexample}. The proof of Theorem~\ref{thm:AVexample} is carried out in Section~\ref{sect:AVexample}, using a modification
  of our main construction. In Section~\ref{sect:bakerflaw} we discuss the error in Baker's original proof, and then prove Theorem~\ref{thm:itdoeshold} in Section~\ref{sect:itdoeshold}. Finally, in Section~\ref{sect:appendix} we give details of the papers and results that are affected by the flawed proofs mentioned earlier.

\subsection*{Notation and terminology} 
The (Euclidean) open ball of radius $r>0$ around $a\in\C$ is denoted by
\[ B(a, r) = \{ z\colon |z - a| < r\}. \]

If $X, Y \subset \C$ are sets such that $X$ lies in a bounded component of the complement of $Y$, then we say that $Y$ \emph{surrounds} $X$. 

Suppose that $X$ is a surface (or any topological space), and let $A\subset X$. We denote the closure of $A$ by $\overline{A}$ or $\closure(A)$.
  If $B\subset A$  is connected, the connected component of $A$ containing $B$ is denoted by $\comp_B(A)$.
    When $x\in A$, we write $\comp_x(A) \defeq \comp_{\{x\}}(A)$.
Finally we set $\Ch \defeq \C \cup \{\infty\}$.

\subsection*{Acknowledgements}
We thank Alex Eremenko for making us aware of the error in Baker's proof discovered by Duval, 
  and for many interesting and useful discussions in the preparation of this article.
We also thank Julien Duval for allowing us to include a description of his argument (see Figure~\ref{fig:duval}), 
  and for reading an early draft of our paper.
	Finally, we thank the referee for several helpful comments, suggestions, and
	   corrections.
\section{An overview of the construction}
\label{sect:overview}
Recall that our main result is Theorem~\ref{thm:maintopological}. For a vanilla function $f$, this theorem asserts the existence of 
two disjoint simply-connected domains $U$ and $V$, each with connected preimage. Theorems~\ref{thm:easyexample}, \ref{thm:meroexample} and \ref{thm:classbexample} follow quickly once Theorem~\ref{thm:maintopological} is established.

Theorem~\ref{thm:maintopological} will be proved by an explicit, but rather complicated, construction of the domains $U$ and $V$. 
  For the benefit of the reader, we give a rough outline of our strategy in this section. We stress that this sketch is not intended to be precise. We note also that more general 
   constructions are possible, but these have additional complexity, which we seek to avoid; the principal goal of this paper is 
   to establish the three examples in Theorems~\ref{thm:easyexample}, \ref{thm:meroexample} and \ref{thm:classbexample}.

Recall from the introduction that a \emph{vanilla} function has a set of critical values with particularly simple properties. It follows from these
   properties any two vanilla functions are in fact topologically equivalent, and hence have the same combinatorial structure. One way of
   expressing this combinatorial structure is as follows (see Proposition~\ref{prop:Dn}); there is an increasing sequence of Jordan domains $(D_n)_{n=1}^{\infty}$ such that each $D_n$ contains exactly $n$ critical values, and such that there is a preimage component $\tilde{D}_n$ of $D_n$ that contains all the critical preimages of these critical values. 

The domains $U$ and $V$ will be obtained as increasing unions of simply-connected domains $U_k, V_k\subset D_{n_k}$, where $n_k$ is a sequence tending to infinity 
   and the $D_n$ are the Jordan domains mentioned above. Here the domains $U_k$ and $V_k$ are disjoint and have the 
   following straightforward relationship to the critical values; there are 
   connected components $\tilde{U}_k, \tilde{V}_k \subset \tilde{D}_{n_k}$ of $f^{-1}(U_k)$ and $f^{-1}(V_k)$, respectively, such that
   \[ f(\CP(f)\cap (\tilde{U_k}\cup \tilde{V_k})) = \CV(f)\cap D_{n_k}. \]
    (That is, all critical points corresponding to critical values in $D_{n_k}$ belong to either $\tilde{U}_k$ or $\tilde{V}_k$.) 
        We shall call $(D_{n_k}, U_k, V_k)$ a \emph{partial configuration}. Compare Proposition~\ref{prop:picture} and Figure~\ref{fig:1}.

Given a partial configuration $(D_{n_k}, U_k, V_k)$, the crux of the proof is to construct a subsequent partial configuration $(D_{n_{k+1}}, U_{k+1}, V_{k+1})$ with 
  $n_{k+1}> n_k$, $V_{k+1} = V_k$, and, most importantly, 
   such that $\tilde{U}_{k+1}$ contains all the preimage components of $U_k$ in $\tilde{D}_{n_k}$.

This construction is given in Proposition~\ref{prop:inductivestep}; compare Figures~\ref{fig:3} and \ref{fig:5}. We choose $n_{k+1}$ so that $D_{n_{k+1}} \setminus D_{n_k}$ contains the same number of critical values as $V_k$.
We then create $U_{k+1}$ by adding a thin ``snake'' to $U_k$ that ``wraps around'' all these critical values and also $V_k$, and then includes the critical values. It is shown that this domain $U_{k+1}$ has the required properties. Proving that there is a preimage component of $U_{k+1}$ that does indeed contain all the preimage components of $U_k$ is the point at which we rely on the simple properties of the preimages of partial configurations.

At the next stage we reverse the roles of $U_k$ and $V_k$, and then iterate the two steps of this process infinitely often. Finally we set $U = \bigcup U_k$ and $V = \bigcup V_k$. It is then shown to follow from the construction 
that $f^{-1}(U)$ and $f^{-1}(V)$ are indeed connected, and this completes the proof of Theorem~\ref{thm:maintopological}.

%
%
%
%

\section{Preliminary results}
\label{sect:preliminary}

We require two simple results concerning branched coverings. The first we use frequently, without comment.
 \begin{prop}[Preimages of simply-connected domains]
    Let $f\colon X\to Y$ be a branched covering between non-compact 
      simply-connected surfaces.  
    Suppose that 
     $U\subset Y$ is a simply-connected
       domain, and  that  $\tilde{U}$ is a 
     component of $f^{-1}(U)$ such that $\tilde{U}\cap \CP(f)$ is finite. 
     Then 
     $f\colon \tilde{U}\to U$ is a proper map, 
      and $\tilde{U}$ is simply-connected.
    
     If additionally $U$ is bounded by a Jordan curve in $Y$ that contains 
      no critical values of $f$, then 
      $\tilde{U}$ is also bounded by a Jordan curve in $Y$.
  \end{prop}
 \begin{remark}
   Note that any non-compact simply-connected surface is homeomorphic 
       to the plane. In particular, $f$ is topologically equivalent either to
       an entire function $\C\to\C$ or a holomorphic map $\D\to\C$.
  \end{remark}
 \begin{proof}
  The fact that any pre-image component of a simply-connected domain 
    is simply-connected follows from the fact
    that $f$ is an open mapping. If $f\colon \tilde{U}\to U$ was not a proper map,
    and hence had infinite degree, it would have to contain infinitely
    many critical points, essentially by the Riemann-Hurwitz formula. 
    See \cite[Proposition 2.8]{walternurialasse}, where an
     analogous result is stated for entire functions; the proof is purely topological and 
      applies equally  in our setting.

  Likewise, the final claim of the proposition is proved for entire functions in 
    \cite[Proposition~2.9~(3)]{walternurialasse}, and again the proof applies
     in our setting.
 \end{proof}

\begin{prop}[Entire functions and branched coverings]
\label{prop:branchedcover}
   Suppose that $f$ is an entire function, and that $U \subset \C$ is a simply-connected domain
   such that $U$ contains no asymptotic values 
	and $U \cap \CV(f)$ is discrete. Then $f$ is a branched covering
   from each component of $f^{-1}(U)$ to $U$.
\end{prop}
\begin{proof}
  This is clear from the definition. 
\end{proof}

We also need two structural results that are useful when studying vanilla functions.

\begin{prop}[Increasing sequence of Jordan domains]
\label{prop:Dn}
  Let $f\colon X\to Y$ be vanilla, and let $c$ be a critical point of $f$. 
    Then there is an increasing sequence $(D_n)_{n=1}^{\infty}$ of Jordan domains
    in $Y$ such that $f(c)\in D_1$, $\bigcup_{n\geq 1} D_n = Y$, and the following hold for all $n\geq 1$:
 \begin{enumerate}[(a)]
    \item $\overline{D_n}\subset D_{n+1}$;
     \item $\partial D_n \cap \CV(f)=\emptyset$; 
     \item $\#(D_n\cap \CV(f)) = n$; and
     \item $f(\tilde{D}_n \cap \CP(f)) = D_n\cap \CV(f)$,
     where $\tilde{D}_n\defeq \comp_c (f^{-1}(D_n))$.
   \end{enumerate}
\end{prop}
\begin{proof}
    Without loss of generality, we can assume that $X=Y=\C$ and that $c=f(c)=0$. 

  \begin{claim}
   Let $R>0$. Then there is a bounded Jordan domain $U\supset B(0,R)$ such that $\partial U \cap \CV(f)=\emptyset$ and 
   such that $\tilde{U}\defeq \comp_0(f^{-1}(U))$ satisfies
     \[ f(\tilde{U}\cap \CP(f)) = U \cap \CV(f). \] 
  \end{claim} 
 \begin{subproof}
  Let $V$ be a ball around $0$ such that
    \[ V\supset \comp_0(f^{-1}(B(0, R)) \cup (f^{-1}( B(0,R)) \cap C(f)). \] 
      Note that this is possible since $B(0,R) \cap \CV(f)$ is finite, and each point in this set
      only has one preimage in $\CP(f)$.
      Let $U^1$ be a Jordan domain containing $\overline{f(V)}$, and $\tilde{U}^1\defeq \comp_0(f^{-1}(U^1))$. 

   For each critical value $v\in U^1$ whose critical preimage is not in $\tilde{U}^1$, choose an arc 
    $\gamma_v\subset U^1 \setminus \overline{B(0,R)}$ 
    connecting $v$ to $\partial U^1$, in such a way that $\gamma_v$ contains no other critical values and
    arcs for different critical  values
    are disjoint. Form the simply-connected domain
    $U^2 \defeq U^1\setminus \bigcup \gamma_v \supset \overline{B(0,R)}$ and set
     $\tilde{U}^2\defeq \comp_0(f^{-1}(U^2))$. Note that, for each critical value $v$ chosen as above, any preimage of $v$ in $\tilde{U}^1$ is not a critical point, and so any preimage component of an 
     arc $\gamma_v$ in $\tilde{U}^1$ 
     is an arc $\tilde{\gamma}_v$ connecting a simple preimage of $v$ to $\partial \tilde{U}^1$. In particular,
     $\tilde{U}^2 = \tilde{U}^1 \setminus \bigcup \tilde{\gamma}_v$.
     
    Then 
    $U^2$ contains $\overline{B(0,R)}$, and the preimage component $\tilde{U}^2$ of $f^{-1}(U^2)$ contains
     all critical preimages of critical values in $U^2$. 
    Choosing a Jordan domain $U\subset U^2$ slightly smaller than $U^2$, if necessary, we  can ensure that
     $\partial U \cap  \CV(f) = \emptyset$. This completes the proof of the claim.
 \end{subproof}
  
    Now we use the claim to construct a subsequence $(D_{n_j})$ of the desired sequence $(D_n)$ inductively,
    as follows. Set $n_1 \defeq 1$, and let $D_1$ be any small disc around $0$ not containing any other
    critical values in its closure. If $n_j$ and $D_{n_j}$ have been defined, let $R\geq j$ be sufficiently large that 
     $\overline{D_{n_j}}\subset B(0,R)$. Then apply the claim to obtain a domain $U\supset \overline{D_{n_j}}$. 
     Set $n_{j+1}\defeq \# (U\cap \CV(f))$, and $D_{n_{j+1}}\defeq U$. Then the domains
    $(D_{n_j})$ satisfy all requirements in the statement of the proposition. It remains to define
     $D_n$ for the remaining values of $n$.

    We next construct the domains $D_n$ for $n_j<n<n_{j+1}$ by removing thin slits containing critical values 
    in $U_{n_{j+1}}$. More precisely, suppose that
    $N_1,N_2$ are such that $N_1<N_2-1$, $D_{N_1}$ and  $D_{N_2}$ have been
    defined, but $D_n$ has not yet been defined for $N_1<n<N_2$. Set $n\defeq N_2-1$, and construct
    $D_n$ from $D_{N_2}$ as follows. For each critical  value $v$ in  $D_{N_2}\setminus D_{N_1}$, let
    $\gamma_v$ again be an arc connecting $v$ to  $\partial D_{N_2}$, not intersecting $\partial D_{N_1}$, and
    such  that different arcs are pairwise disjoint. 

    For each such $v$, 
    let $\tilde{\gamma}_v = \comp_{c(v)}(f^{-1}(\gamma_v))$, where $c(v)$ is the unique critical point of $f$ 
    over $v$. Then $\tilde{\gamma}_v$ is a cross-cut of $\tilde{D}_{N_2}$, not intersecting $\tilde{D}_{N_1}$. 
    At least one  of these crosscuts, say $\tilde{\gamma}_{v_0}$, does not separate 
       $\tilde{D}_{N_1}$ from
    any other $\tilde{\gamma}_v$. Observe that any other component of $f^{-1}(\gamma_{v_0})$ in $\tilde{D}_{N_2}$ is an arc
    connecting $\partial \tilde{D}_{N_2}$ to some simple preimage of $v_0$, and hence does not 
    disconnect $\tilde{D}_{N_2}$. 

    Set $U\defeq D_{N_2}\setminus \gamma_{v_0}$ 
    and $\tilde{U}\defeq \comp_{\tilde{D}_{N_1}}(f^{-1}(U))$. 
    Then, by construction, $\tilde{U}$ contains $n=N_2-1$ critical points.
    Slightly shrinking $U$ if necessary, we obtain a disc $D_n$ with $D_{N_1}\subset D_n\subset D_{N_2}$, satisfying the
    desired properties. Proceeding inductively, we define $D_n$ for all $n\geq 1$. 
\end{proof}

\begin{prop}[Connecting to the critical point]\label{prop:alpha}
 Let $f\colon X\to Y$ be vanilla and let 
 $(D_n)_{n=1}^{\infty}$ 
    and $(\tilde{D}_n)_{n=1}^{\infty}$ be sequences as in the statement of Proposition~\ref{prop:Dn}.
    Fix 
    $n\geq  1$ and $\zeta\in \partial \tilde{D}_n$.

   Then there is a Jordan arc $\tau$ connecting $f(\zeta)$ to $\partial D_{n+1}$ and
     lying in the annulus  $D_{n+1}\setminus \overline{D_n}$ apart from its endpoints, such that $\comp_{\zeta}(f^{-1}(\tau))$ contains the unique critical point $c'$ of $f$ in
      $\tilde{D}_{n+1}\setminus \tilde{D}_n$.  
\end{prop}
\begin{proof}
   Note that $f\colon \tilde{D}_n\to D_n$ has degree $n+1$, while $f\colon \tilde{D}_{n+1}\to D_{n+1}$
    has degree $n+2$, by the Riemann-Hurwitz formula. Hence $\tilde{D}_{n+1}$ contains 
    exactly one component $U$ of $f^{-1}(D_n)$ different from $\tilde{D}_n$, and $U$
    is mapped conformally to $D_n$. 

  Choose an arc $\tau_0$ connecting $f(\zeta)$ to $\partial D_{n+1}$ 
    lying in the annulus ${D}_{n+1}\setminus \overline{{D}_n}$
    and passing through the critical 
    value $f(c')$. Then $\comp_{c'}(f^{-1}(\tau_0))$ contains two points 
     of $f^{-1}(f(\zeta))$. One of these points must be the unique preimage of $\zeta$ on $\partial U$, while
     the other is some point 
    $\zeta_0\in \Xi \defeq f^{-1}(f(\zeta)) \cap \partial  D_n$.

   Now consider what happens when we apply $j$ Dehn twists to $\tau_0$ near 
      $\partial D_{n}$, 
     for $0 \leq j\leq n-1$, obtaining curves $\tau_j$. Each $\tau_j$ has the same properties as $\tau_0$ above; 
     let $\zeta_j$ be the unique element of 
     $\Xi$ in  $\comp_{c'}(f^{-1}(\tau_j))$. 
     Since  $f\colon \partial \tilde{D}_{n}\to \partial D_{n}$ is a 
     degree $n$ covering map of circles, the values of $\zeta_j$ cycle through all the elements of 
     $\Xi$. Hence, there is some $j$ such that $\zeta_j = \zeta$, and we can take $\tau=\tau_j$.
\end{proof}

 As mentioned in the previous section, it follows that any two vanilla functions are topologically
    equivalent. Although we will not use this fact directly in the following, we give a proof for completeness.
		
\begin{cor}[Topological uniqueness]
  Suppose that $f^1\colon X^1\to Y^1$ and $f^2\colon X^2\to Y^2$ are vanilla.
    Then there are homeomorphisms $\phi\colon X^1\to X^2$ and  $\psi\colon Y^1\to Y^2$ such that
    $\psi\circ f^1 = f^2 \circ \phi$.
\end{cor}
\begin{proof}
  Let $(D^j_n)_{n=1}^{\infty}$ and $(\tilde{D}^j_n)_{n=1}^{\infty}$ be as in Proposition~\ref{prop:Dn},
   for $j=1,2$. We inductively specify homeomorphisms
       \[ \psi_n \colon \closure(D_n^1)\to \closure(D_n^2)
       \quad\text{and}\quad 
       \phi_n \colon \closure(\tilde{D}_n^1)\to \closure(\tilde{D}_n^2)\]
    such that 
   \begin{equation}\label{eqn:equivalence}
     \psi_n \circ f^1 = f^2 \circ \phi_n,
   \end{equation}
    $\psi_{n+1}|_{\closure(D_n^1)} = \psi_n$, and 
    $\phi_{n+1}|_{\closure(\tilde{D}_n^1)} = \phi_n$. The claim follows by defining 
    $\psi$ and $\phi$ to be the common extension of the maps $\psi_n$ and $\phi_n$, respectively. 

   Let $\psi_1\colon \closure(D_1^1)\to \closure(D_1^2)$ be any homeomorphism that fixes $0$,
     and let $\phi_1$ be either of the two lifts of $\psi_1$, such that $\psi_1\circ f^1 =  f^2\circ \phi_1$. 
 
    Now suppose that the homeomorphisms $\psi_n$ and $\phi_n$ have already been defined.
     Fix any point $\zeta^1_n\in \partial \tilde{D}_n^1$, and set 
     $\zeta^2_n\defeq \phi_n(\zeta^1_n)$. Applying Proposition~\ref{prop:alpha} to 
     $f^1$ (with 
     $\zeta = \zeta^1_n$) and to $f^2$ (with $\zeta=\zeta^2_n$), we obtain Jordan arcs $\tau_n^1$ and $\tau_n^2$. 
     We extend $\psi_n$ to a homeomorphism $\psi_{n+1}$ from
     $\closure(D_{n+1}^1)$ to $\closure(D_{n+1}^2)$, in such a way that $\psi_{n+1}$ maps $\tau_n^1$ to $\tau_n^2$,
     taking the critical value of $f^1$ in $\tau_n^1$ to the critical value of $f^2$ in $\tau_n^2$. 
     It is now straightforward to verify that there is a lift $\phi_{n+1}$, extending $\phi_n$, 
     satisfying~\ref{eqn:equivalence}. 
\end{proof}

%
%
%
%

\section{Proof of Theorem~\ref{thm:maintopological}}
\label{sect:construction}

As mentioned in Section~\ref{sect:overview}, we consider certain triples of Jordan domains to 
   facilitate the construction. 
These are defined as follows.

\begin{defn}[Partial configuration]
\label{def:partialconf}
  Let $f\colon X\to Y$ be vanilla. A triple $(D,U,V)$ of Jordan domains in $Y$ 
    is called a \emph{partial configuration} for $f$ if all 
the following hold:
   \begin{enumerate}[(a)]
     \item $\overline{U}\cap \overline{V}=\emptyset$;
      \item $\overline{U} \cup \overline{V} \subset D$;
      \item $D\cap \CV(f)\subset U\cup V$; 
      \item there is a component $\tilde{D}$ of $f^{-1}(D)$ such that 
          $f(\tilde{D}\cap \CP(f))=D\cap \CV(f)$;
      \item there are components $\tilde{U}$ and $\tilde{V}$ of $f^{-1}(U)$ and $f^{-1}(V)$, respectively, such that
            $\tilde{D}\cap \CP(f) \subset \tilde{U}\cup\tilde{V}$. 
   \end{enumerate}
\end{defn}

These partial configurations are easy to work with because their preimages have 
  a very specific and extremely simple structure, as described in the following proposition. 
  This structure is illustrated in Figure~\ref{fig:1}. In this proposition, and subsequently, if 
  $U$ is a domain, then we denote by $m_U$ the number of critical values of $f$ in~$U$.
  (Of course $m_U=m_U(f)$ depends on the function $f$, but we suppress $f$ from the notation since it will be fixed whenever this notation is used.)

\begin{prop}[Structure of partial configurations]
\label{prop:picture}
  Let $f\colon X\to Y$ be vanilla, and suppose that  $(D,U,V)$ is a partial configuration. 
   Let $\gamma$ be a crosscut of $D$ that
   separates $\overline{U}$ from $\overline{V}$. Let $z$ and $w$ be its endpoints on $\partial D$, labeled such that
    $U$ is on the left of $\gamma$  and  $V$ is  on the right of  $\gamma$ when $\gamma$ is oriented from
    $z$ to~$w$.  

    Label the 
    preimages $\tilde{z}_1,\dots, \tilde{z}_d$ of $z$ on  $\partial \tilde{D}$ 
    and the preimages $\tilde{w}_1,\dots,\tilde{w}_d$ of $w$ on  $\partial \tilde{D}$
    in positive orientation such that $\tilde{w}_1$ is between $\tilde{z}_1$ and $\tilde{z}_2$; here $d = m_U + m_V + 1$ is the degree of $f$ on $\tilde{D}$. This determines
    the labeling up to  the choice of $\tilde{z}_1$.

    Then the choice of $\tilde{z}_1$ can be made in  such  a way  that the 
    preimage components of $f^{-1}(\gamma)\cap \tilde{D}$ connect these preimages as follows.  
  \begin{enumerate}[(a)]
    \item $\tilde{z}_1$ is connected to $\tilde{w}_{m_V+1}$. 
    \item For $j=1,\dots,m_V$, $\tilde{w}_j$ is connected to  $\tilde{z}_{j+1}$. \label{item:Ucomponents}
    \item For  $j=m_V+2,\dots,d$, $\tilde{z}_j$ is connected to  $\tilde{w}_j$.\label{item:Vcomponents} 
    \item Each component of $\tilde{D}\setminus f^{-1}(\gamma)$ (which we shall call a \emph{face}) 
      contains either exactly one component of 
      $f^{-1}(U)$ or exactly one component of $f^{-1}(V)$. 
    \item The faces containing 
      $\tilde{U}$ and $\tilde{V}$ are adjacent to each other, and separated by the arc connecting
      $\tilde{z}_1$ and $\tilde{w}_{m_V+1}$. More precisely, the face containing $\tilde{U}$ is 
       bounded by this arc and those in~\ref{item:Vcomponents}, while that containing $\tilde{V}$ is
       bounded by this arc and those in~\ref{item:Ucomponents}. 
     \item Each connected component of $f^{-1}(U)$ in $\tilde{D}$, apart from $\tilde{U}$, is contained in a face bounded 
        solely by one of the arcs from~\ref{item:Ucomponents}. Similarly, each component
     of $f^{-1}(V)$ in $\tilde{D}$, apart from  $\tilde{V}$, is contained in a face bounded by one  of the 
      arcs from~\ref{item:Vcomponents}.
  \end{enumerate} 
\end{prop}

\begin{figure}
	\includegraphics[width=\textwidth]{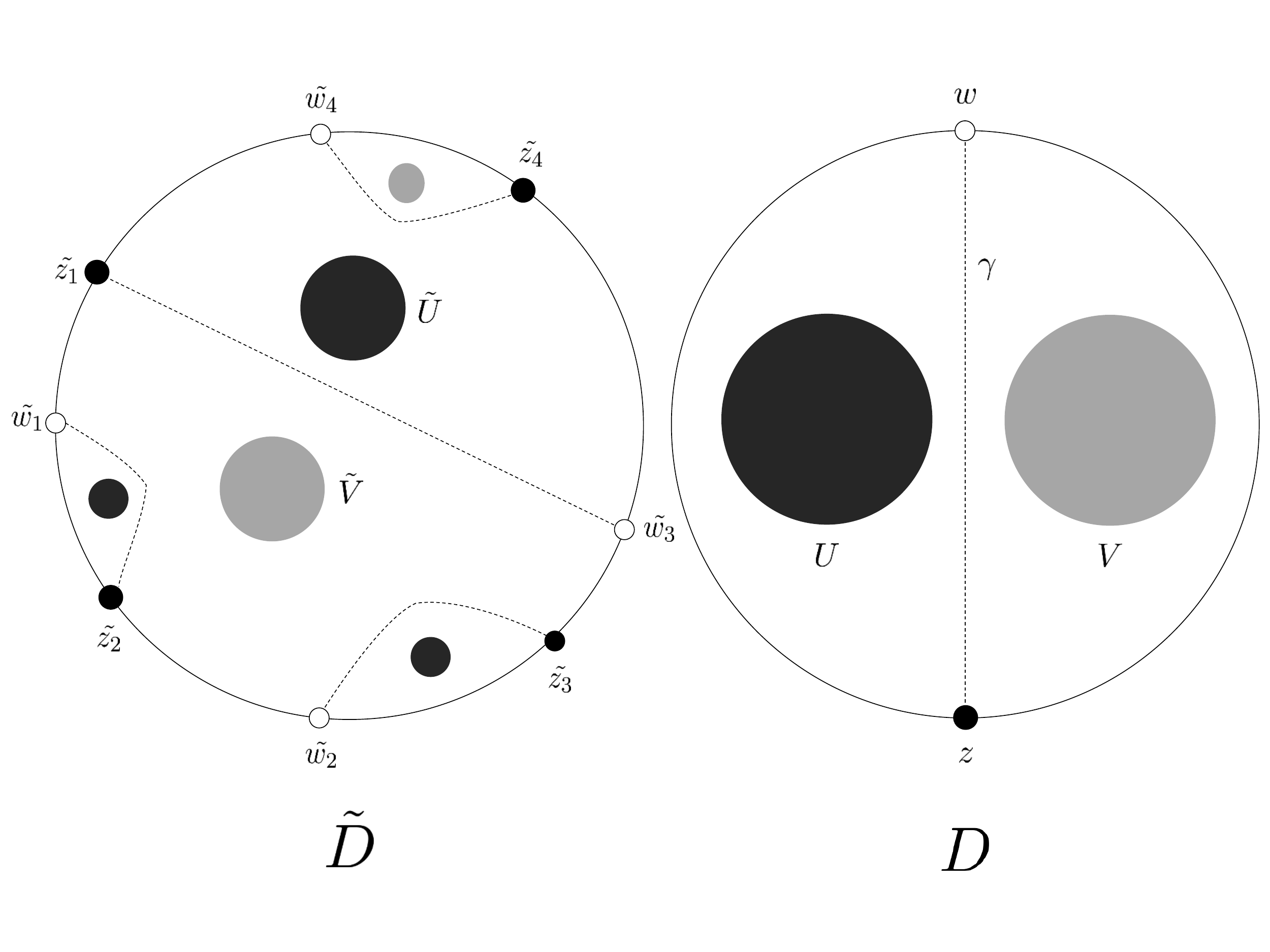}
  \caption{The initial preimage (left) and image (right) in Proposition~\ref{prop:picture}. 
   Note that the domains $U$ and $V$ may look very complicated geometrically, but topologically we are
     speaking simply of a configuration of two discs within a larger disc, as shown here. 
   Black discs on the left are preimage components of $U$, while grey discs are preimage components of
      $V$. For this example we have taken $m_U=1$ and $m_V=2$. We stress that, in general, $D$ will have other preimage components; these are not illustrated.}\label{fig:1}
\end{figure}

\begin{remark}
  We  allow the case where $m_U=0$ or $m_V=0$, in  which case 
    the preimage 
    domains $\tilde{U}$ and $\tilde{V}$ from Definition~\ref{def:partialconf} may not be unique. In this case,
    the claim should be understood as follows: for all valid choices of  $\tilde{U}$ and  $\tilde{V}$, there is a 
    choice of $\tilde{z}_1$ satisfying the properties listed.
\end{remark}
\begin{proof}
  Recall that all critical values of $f$ in $D$ lie in $U$ or $V$. Let $D_L$ be the component of $D\setminus \gamma$
     containing $U$. Then $A\defeq D_L\setminus \overline{U}$ is an annulus containing no singular values.
     Every component $\tilde{A}$ of $f^{-1}(A)$ is mapped as a finite-degree covering map, and hence
      is also an annulus. It follows that $\tilde{A}$ is a face with exactly one component of
      $f^{-1}(U)$ removed. Clearly we can apply the same argument to $V$, replacing $D_L$ by
      the other component $D_R$ of $D\setminus \gamma$.  Let us call preimage components of $D_L$ \emph{L-faces},
     and other faces \emph{R-faces}. 

     By the above, every face contains exactly one preimage component of $U$ or of $V$, and is mapped
     with the same degree as this component. It follows that the set $\tilde{U}$ from 
     Definition~\ref{def:partialconf} is contained in an  L-face of degree $m_U+1$, and there
     are $m_V$ further L-faces, all of which are simple. Similarly, $\tilde{V}$ is contained in 
     a an R-face of degree 
     $m_V+1$, and there are a further $m_U$ R-faces, all simple. 

   Consider the dual graph $G$ to this picture, where each face represents a vertex, and two faces
    are connected if they are adjacent; in other words if they have a common component of $f^{-1}(\gamma)$ in their boundary.
    By the remarks above, $G$ is a connected graph with $d$ edges and $m_U + m_V+2=d+1$ vertices. So $G$ is a tree. 
     Moreover, $G$ has one vertex of
     degree $m_U+1$ corresponding to  $\tilde{U}$, one of degree $m_V+1$, corresponding to  $\tilde{V}$, 
     and a further $d-1$ vertices, all of which are leaves (vertices of degree $1$). 

    Since $G$ is connected, the vertices representing $\tilde{U}$ and  $\tilde{V}$ must be connected by 
     a simple path in $G$, and since all other vertices are leaves, this path must in fact be an
     edge. This edge corresponds to some component of $f^{-1}(\gamma)$. We choose $\tilde{z}_1$ to be the
     preimage of $z$ contained in this component. It then follows that the description is indeed as above.
\end{proof}

We will deduce Theorem~\ref{thm:maintopological} from the following proposition, which is the crux of our construction.
\begin{prop}[Extending partial configurations]\label{prop:inductivestep}
  Let $f\colon X\to Y$ be vanilla, and let $(D_n)_{n=1}^{\infty}$ be the sequence from Proposition~\ref{prop:Dn}.  Suppose that $n\geq 1$ and that $(D_n, U, V)$ is a partial configuration for $f$. Set $n' = n + m_V$.
   
  Then there exists $U' \supsetneq U$, such that $(D_{n'} , U' , V)$ is 
    a partial configuration, and such that $\tilde{U}'$ contains all the connected components of 
     $f^{-1}(U)\cap \tilde{D}_n$.
\end{prop}
\begin{proof}
   For simplicity of notation, set $m\defeq m_V$. 
	Observe that if $m=0$, then $\tilde{U}$ is connected by Proposition~\ref{prop:picture}, and there is nothing
    to prove. So we may assume that $n'>n$. We must describe how $U$ is extended to $U'$.

      We apply Proposition~\ref{prop:picture} to 
      the partial configuration $(D_n,U,V)$, and use the notation given there in the following paragraphs. 
      We begin by constructing a curve $\alpha$ connecting the point $w\in\partial D_{n}$ to $\partial D_{n'}$
      in such a way that there is a curve contained in $f^{-1}(\alpha)$ connecting the point 
      $\tilde{w}_{m+1}$
      to $\partial \tilde{D}_{n'}$ and passing through all the critical points in 
      $\tilde{D}_{n'}\setminus{\tilde{D}_{n_{k}}}$.
      This construction proceeds as follows. First let $\tau$ be the arc obtained from Proposition~\ref{prop:alpha},
      with $\zeta= \zeta_1 \defeq \tilde{w}_{m+1}$. Then set $\alpha_1 = \tau$.
      
      Let $\omega_1$ be  the endpoint of  $\tau$ on  $\partial D_{n+1}$. Recall that all critical points
      of $f$ are simple. It follows that the component of $f^{-1}(\tau)$ 
      containing $\zeta_1$ contains two points of $f^{-1}(\omega_1)$, say $\tilde{\omega}_1^1$ and
      $\tilde{\omega}_1^2$. One of the two arcs of $\partial \tilde{D}_{n+1}\setminus\{\tilde{\omega}^1_1,\tilde{\omega}^2_1\}$ 
      does not contain any  other points of $f^{-1}(\omega_1)$. We may assume that the points are labeled such that this arc, when 
      oriented from $\tilde{\omega}_1^1$ to $\tilde{\omega}_1^2$, traverses $\partial \tilde{D}_{n+1}$ in positive orientation. 
      We now set $\tilde{\omega}_1 \defeq \tilde{\omega}_1^2$, and apply Proposition~\ref{prop:alpha} again, 
      with $\zeta=\zeta_2 \defeq \tilde{\omega}_{1}$, obtaining a curve $\alpha_2 = \tau$. 
      
      Continuing inductively, we construct a curve $\alpha=\alpha_1\cup\dots\cup \alpha_{m}$ connecting 
      $\partial D_{n}$ to $\partial D_{n'}$, and passing through the $m$ critical values in 
      $D_{n'}\setminus D_{n}$; see the right-hand side of Figure~\ref{fig:3}. This completes the
      construction of the curve $\alpha$.

      \begin{figure}
      	\includegraphics[width=\textwidth]{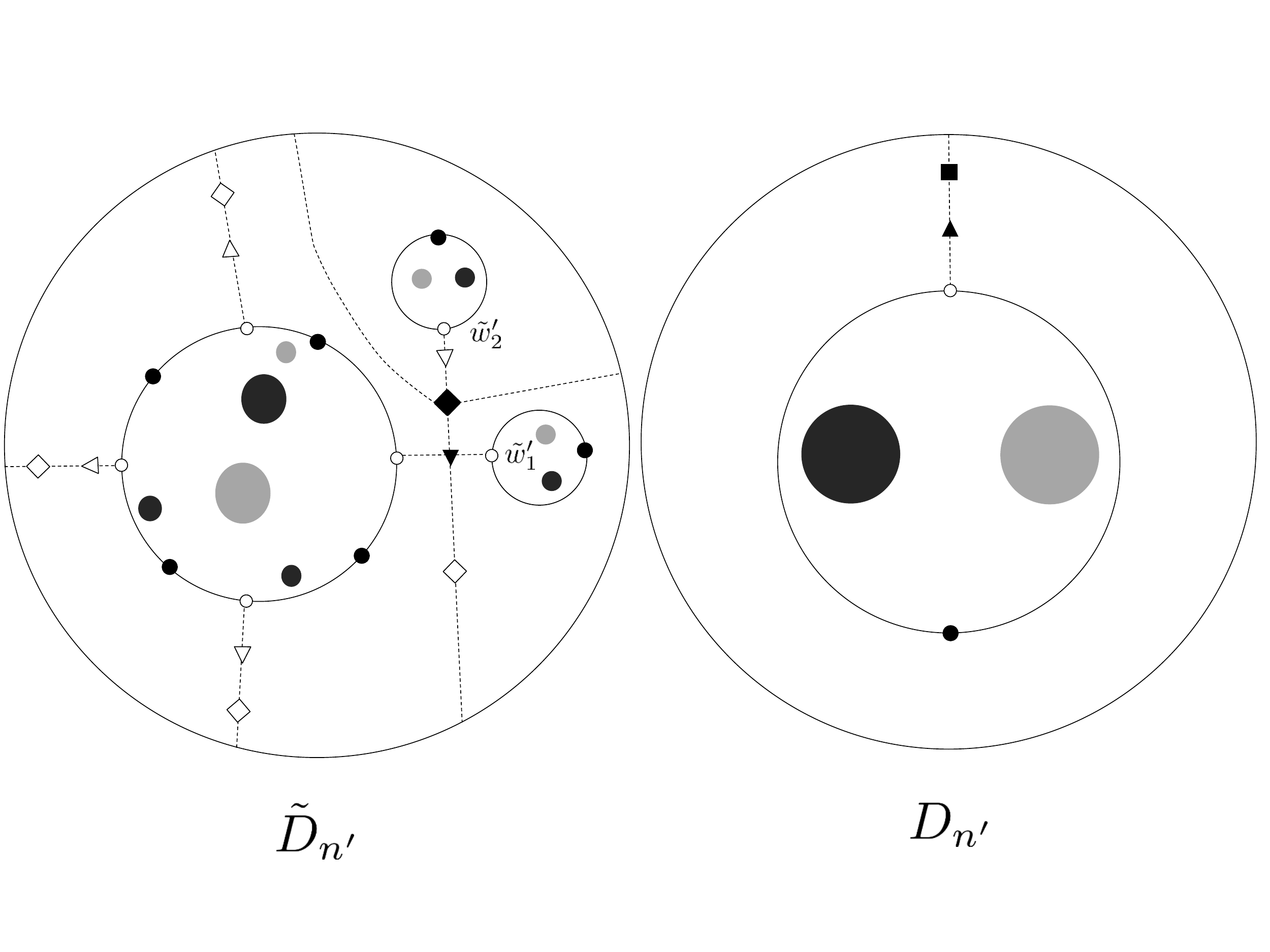}
        \caption{Application of Proposition~\ref{prop:inductivestep} to the configuration from Figure~\ref{fig:1}. 
    The larger disc $D_{n'}$ in the image (right) 
       has two additional critical points, shown as a triangle and a square, and a curve $\alpha$ indicated as a dashed line. 
       The preimage $\tilde{D}_{n'}$ (left) shows preimages of the critical values as triangles and squares, 
       with the solid figure indicating the preimage that is actually the critical point. 
        The preimage of the curve $\alpha$ is indicated as a dashed line. Note also two additional preimages of $w$, labeled $\tilde{w}'_{1}$ and $\tilde{w}'_{2}$, 
        each on the boundary of a new preimage of the original disc.}
			\label{fig:3}
      \end{figure}

      Let $v_1,\dots,v_{m}$ be  the critical values on $\alpha$, ordered as they are encountered
      when traversing $\alpha$ from $w$ to $\partial D_{n'}$. Let $c_1,\dots,c_{m}$ be the 
      corresponding critical points. Then, for each $j=1,\dots,m$, there is a simple preimage
      $\tilde{D}_{n}^j$ of $D_{n}$, attached to $c_j$ by a simple preimage of the  
      piece $\beta_j$ of $\alpha$ connecting $D_{n}$ to $v_j$, and with a preimage $\tilde{w}'_{j}$ of $w$ on its boundary; see the left-hand side of Figure~\ref{fig:3}.
			Note, in particular, that $\beta_m$ is an arc in $D_{n'} \setminus D_n$ that connects $w$ to all the critical values $v_1, \ldots, v_m.$
      
    Let $\Gamma_0$ be a Jordan curve in $D_{n'}\setminus (\overline{U}\cup \overline{V}\cup (\beta_m\setminus\{w\}))$ 
    obtained as follows. We begin at $w$, proceed along an arc outside of $D_{n}$ that connects
      $w$ to $z$, running around $\beta_m$ in negative orientation, and returning 
      to $w$ from $z$ via the curve $\gamma$ from Proposition~\ref{prop:picture}. So $\Gamma_0$ 
      surrounds $V$ and $\beta_m\setminus\{w\}$, but not $U$. 

    Then the preimage of $\Gamma_0$ consists of $m_{U}$ simple preimages of $\Gamma_0$, and one loop that is mapped by $f$ as a degree
      $2m + 1$ covering of circles. Each of the simple preimages contains the component of  $f^{-1}(\gamma)$ connecting
     $\tilde{z}_j$ and  $\tilde{w}_j$ for some
      $j\in \{m+2, \dots , n+1\}$, and is a loop surrounding the corresponding simple R-face. The 
      remaining loop passes (in positive orientation) through
        $\tilde{w}_{m+1}$, $\tilde{z}_1$, $\tilde{w}_{1}$, $\tilde{z}_2$, $\tilde{w}_2$,  \dots , $\tilde{w}_m$,
        $\tilde{z}_{m+1}$,  
        and then through the components $\tilde{D}_{n}^1$, \dots,  $\tilde{D}_{n}^m$. 
    
    Now consider an arc $\Gamma$, not intersecting $\overline{V} \cup \CV(f)$ and not intersecting  $\overline{U}$ except in
      one endpoint,  defined as follows. The arc starts at some point $\mu\in\partial U$, 
      and runs around $\beta_m$ and $\overline{V}_k$ $m$ times in negative orientation; in other words, 
      in the same manner as $\Gamma_0$. On its last loop it enters
       $D_{n}$ at
      the point $z$, traversing along the arc $\gamma$ and ending at $w$. Hence
      $\Gamma$ is homotopic (in $D_{n'}\setminus \CV(f)$) to a curve 
      that connects $\mu$ to $w$ within the left half of $D_{n}\setminus\gamma$, and then
      traverses $\Gamma_0$ $m$ times. 

   Since $\Gamma$ does not
      contain any critical  values, every component of $f^{-1}(\Gamma)$ is an  arc beginning on the boundary
      of some preimage component of $U$, and ending in a preimage of $w$. 
      Let $j\in \{1,\dots,m\}$ and consider the component $\tilde{\Gamma}_j$ of 
      $\Gamma$ that starts on the boundary of the simple preimage of $U$ contained in the 
      L-face bounded by an arc connecting $\tilde{w}_j$ and  $\tilde{z}_{j+1}$. Our 
      discussion of the structure of $f^{-1}(\Gamma_0)$ 
      shows that, for $j<m$,  $\tilde{\Gamma}_j$ ends at the preimage $\tilde{w}'_{j+1}$ of $w$ on the boundary  of the 
       disc $\tilde{D}_{n}^{j+1}$, while $\tilde{\Gamma}_m$ ends at $\tilde{w}_{m+1}$. 
      Furthermore, if $\tilde{\Gamma}$ is the component of $f^{-1}(\Gamma)$ 
       ending at $\tilde{w}'_1$, then $\tilde{\Gamma}$ begins on $\tilde{U}$. 

       \begin{figure}
      	\includegraphics[width=\textwidth]{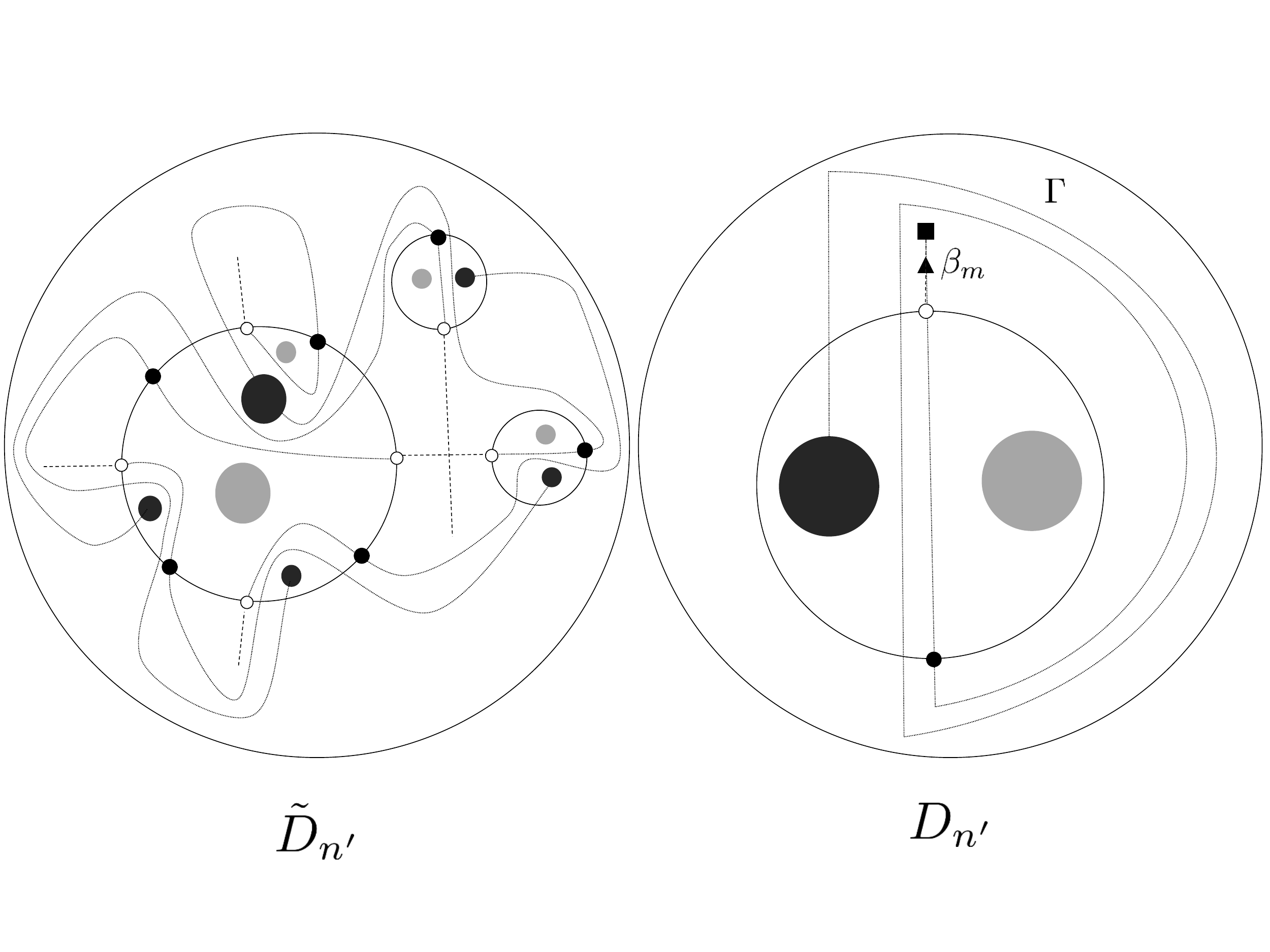}
        \caption{The thin channel $W$ has been added to the image (right) from Figure~\ref{fig:3}; the part containing $\Gamma$ is solid, and the part containing $\beta_m$ is dashed. 
      The preimage of $W$ has been added to the preimage on the left, shown solid or dashed appropriately}\label{fig:5}
      \end{figure}

     Recall from the construction  of  $\alpha$ that
         $\tilde{w}'_1$, \dots,  $\tilde{w}'_m$ and $\tilde{w}_{m+1}$, as well as the critical points
         $c_1,\dots,c_m$ all belong to the same component $\tilde{\beta}$ of
         $f^{-1}(\beta_m)$. Take a thin channel $W$ containing the Jordan arc $\Gamma \cup \beta_m$, in such
          a way that $U' \defeq U \cup W$ is simply-connected; see the right-hand side of Figure~\ref{fig:5}. 
          
          It then follows that the component of $f^{-1}(U')$ containing 
           $\tilde{w}_{m+1}$ contains the connected set
               \[ \tilde{\beta} \cup \tilde{\Gamma} \cup \bigcup_{j=1}^m \tilde{\Gamma}_j \cup  (f^{-1}(U)\cap \tilde{D}_{n}); \]
          see the left-hand side of Figure~\ref{fig:5}. This completes the proof.
\end{proof}

\begin{cor}[Sequence of partial configurations]
\label{cor:mainconstruction}
  Let $f\colon X\to Y$ be vanilla, and let $(D_n)_{n=1}^{\infty}$ be the sequence from Proposition~\ref{prop:Dn}. Then there is a sequence 
     $(D_{n_k}, U_k , V_k)_{k=1}^{\infty}$ of partial configurations such that 
    $(U_k)$ and $(V_k)$ are increasing sequences. Furthermore,
    the sequences can be chosen such that every component of $f^{-1}(U_1)$ is contained in
     $\tilde{U_k}$ for sufficiently large $k$, and similarly for the components of $f^{-1}(V_1)$.
\end{cor}
\begin{proof}
   We construct a sequence $(D_{n_k},U_k,V_k)_{k=1}^{\infty}$ of partial configurations, with $n_k\to\infty$,  such that when $k$ is odd we have
\begin{equation}\label{eqn:inductivepropertyodd}
U_k\subsetneq U_{k+1}, V_k = V_{k+1}, \text{ and } f^{-1}(U_k)\cap \tilde{D}_{n_k} \subset \tilde{U}_{k+1},
\end{equation}
  and when $k$ is even we have
\[
V_k\subsetneq V_{k+1}, U_k =  U_{k+1}, \text{ and } f^{-1}(V_k)\cap \tilde{D}_{n_k} \subset \tilde{V}_{k+1}.
\]

    This sequence is defined recursively using Proposition~\ref{prop:inductivestep}.
 To anchor the recursion, we simply set 
    $n_1\defeq 1$, 
    let $V_1$ be a sufficiently small disc around the critical value of $f$ in $D_1$, and let $U_1$ be 
     any Jordan domain whose closure is contained in $D_1$ and disjoint from $\overline{V_1}$. 
   It is easy to check that this is indeed a partial configuration.

 Now suppose that $(D_{n_k}, U_k,V_k)$ has been defined for $k\geq 1$. If $k$ is odd, we apply
    Proposition~\ref{prop:inductivestep} with $n=n_k$, $U=U_k$ and $V=V_k$. 
   Setting $n_{k+1}\defeq n'$, $U_{k+1}\defeq U'$ and $V_{k+1}\defeq V_k$, we obtain another partial configuration
   satisfying~\ref{eqn:inductivepropertyodd}. 
   If $k$ is even, then we similarly apply Proposition~\ref{prop:inductivestep}, but interchanging the roles
    of $U_k$ and $V_k$. This completes the construction. 

   Let $W$ be a connected component of $f^{-1}(U_1)$. Then there is some $N$ such that 
     $W\subset \tilde{D}_n$ for $n\geq N$. Let $K$ be odd such that $n_K\geq N$. Then 
     \[ W\subset f^{-1}(U_1) \cap \tilde{D}_{n_K} \subset f^{-1}(U_{{k}}) \cap \tilde{D}_{n_k} \subset
               \tilde{U}_k \]
     for $k>K$, as desired (and similarly for $f^{-1}(V_1)$). This completes the proof.
       \end{proof} 
       
The proof of Theorem~\ref{thm:maintopological} is now quite straightforward.
\begin{proof}[Proof of Theorem~\ref{thm:maintopological}]
   Set \[U\defeq\bigcup_{k\in\N} U_k\quad\text{and}\quad V\defeq \bigcup_{k\in\N} V_k,\]
    where $U_k$ and $V_k$ are as in Corollary~\ref{cor:mainconstruction}. Then $U$ and $V$ are  
    increasing unions of simply-connected domains and so are simply-connected.
    Moreover, $U\cap V =  \emptyset$. 

    Let $\tilde{U}$ be the component of $f^{-1}(U)$ containing $\tilde{U}_k$ for all $k$. Then 
     $f^{-1}(U_1) \subset \tilde{U}$ by assumption.  Choose $x\in f^{-1}(U)$, and let $\gamma$ be a curve in 
     $U$ connecting $f(x)$ and some point  $y$ of $U_1$ without passing  through  any critical values.  
     Then there is a curve  $\tilde{\gamma}\subset f^{-1}(U)$ connecting $x$ to a point 
     in $f^{-1}(y)\subset f^{-1}(U_1)\subset \tilde{U}$,  and  hence $x\in  \tilde{U}$.
     So  $f^{-1}(U)=\tilde{U}$ is connected. Likewise $f^{-1}(V)$ is connected. 
\end{proof}

 We also observe that we can 
    ensure that the domains $U$ and $V$ of Theorem~\ref{thm:maintopological} have a common boundary,
    as promised in Remark~\ref{rmk:commonboundary}.

\begin{obs}[Common boundary]\label{obs:commonboundary}
  The domains $U$ and $V$ in Theorem~\ref{thm:maintopological} can be chosen such that $\partial U = \partial V$.
\end{obs}
\begin{proof}
 Let $d$ be a metric on $X$.  
   In the setting of Proposition~\ref{prop:inductivestep}, we claim that the distance
    $d(U', v)$ can be chosen arbitrarily small for all $v\in\partial V$.

  To prove this, we let $\eps>0$, and choose a simply-connected domain 
   $U^1$
     such that
    $U\subset U^1\subset \overline{U^1}\subset D_n\setminus \overline{V}$, and such that every point
     of $\partial V$ has distance at most $\eps$ from  $U^1$. 
   
    Then $(D_n, U^1, V)$ is also a partial configuration, and we can apply
    Proposition~\ref{prop:inductivestep}. The resulting domain  $U'$ has the required properties.

   Now let $(\eps_k)$ be a sequence of positive real numbers tending to zero. In the $k$-th step of the inductive construction  in
     Corollary~\ref{cor:mainconstruction}, apply the above observation for $\eps=\eps_k$. 
   For the resulting domains $U$ and $V$, let $v\in \partial V$, and let $\delta>0$.
     Then, for all sufficiently large $k$, 
     $B(v,\delta)$ intersects $V_k$, and hence 
     $\partial V_k\cap B(v,\delta)\neq\emptyset$. If $k$ is chosen 
     odd and sufficiently large that $\eps_k<\delta$, then it follows that $U_k$ contains 
     a point of distance at most $2\delta$ from $v$. As $v$ and $\delta$ 
     were arbitrary, we have shown $\partial V\subset \partial U$.
     The converse inclusion follows analogously.
\end{proof}

\begin{proof}[Proof of Corollary~\ref{cor:maintopological}]
First we apply Theorem~\ref{thm:maintopological} 
   to obtain two disjoint simply-connected domains $U$ and $V$ each with connected preimage. Since $f$ has
    infinite degree on $f^{-1}(V)$, and $f$ is vanilla,   
    $V \cap \CV(f)$ is infinite. Set $U_1 \defeq U$ and then again apply Theorem~\ref{thm:maintopological} 
     to the restriction $f\colon f^{-1}(V) \to V$.
    We obtain two simply-connected subdomains of $V$, each with connected preimage. The result follows 
    by induction.
\end{proof}

%
%
%
%

\section{Examples}
\label{sect:examples}

In this section, we show how to deduce Theorem~\ref{thm:easyexample}, Theorem~\ref{thm:meroexample} and Theorem~\ref{thm:classbexample} from Theorem~\ref{thm:maintopological}.

\begin{proof}[Proof of Theorem~\ref{thm:easyexample}]
Let $f(z) = e^z + z$. The critical points of 
$f$ are $c_m = (2m+1)\pi i$,  
   with $m\in\Z$, and the corresponding critical values are
   $v_m = (2m+1)\pi i - 1$. 
Furthermore, it is easy  to  see that $f$ has no finite asymptotic values. 
By Proposition~\ref{prop:branchedcover}, $f$ is a branched covering map from $\C$ to $\C$, and by
   the above statements on critical values,  
    $f$ is vanilla.
The result then follows by Corollary~\ref{cor:maintopological}.
\end{proof}

\begin{proof}[Proof of Theorem~\ref{thm:meroexample}]
Let $f(z) = \tan z + z$. Let $U$ be the upper half-plane given by 
   $U \defeq \{ z \colon \im z > 0 \}$, and we note that $f^{-1}(U) = U$. Observe also that $f$ has no asymptotic values, and that $U$ contains no poles of $f$.

The critical points of $f$ in $U$, which are all simple, are 
\[
c_m = \left(m + \frac{1}{2}\right)\pi + i \operatorname{arsinh} 1, \quad\text{for } m \in \Z.
\]
The critical values of $f$ in $U$ are
\[
v_m = c_m + i \sqrt{2}, \quad\text{for } m \in \Z,
\]
and so each critical value has a unique critical preimage. 

By Proposition~\ref{prop:branchedcover}, $f \colon U \to U$ is vanilla, and the result follows 
    by Corollary~\ref{cor:maintopological}.
\end{proof}
\begin{remark}
The dynamics of the function $f$ from Theorem~\ref{thm:meroexample} were studied in \cite{MR3581221}.
\end{remark}

\begin{proof}[Proof of Theorem~\ref{thm:classbexample}]
Bergweiler \cite{walter} introduced the transcendental entire function
\[
f(z) = \frac{12\pi^2}{5\pi^2-48}\left(\frac{(\pi^2-8)z+2\pi^2}{z(4z-\pi^2)}\cos\sqrt{z} + \frac{2}{z}\right).
\]

It was shown in \cite{walter} that $f$ has a completely invariant Fatou component $U$, such that $0 \in \partial U$ and $(0,\infty) \subset U$. It was also shown that $0$ is the only finite asymptotic value of $f$, and that $f$ has infinitely many critical values, all of which lie in a real interval of the form $(0,c) \subset U$, and which accumulate only on the origin. All critical points of $f$ are simple, 
  real and positive, and clearly also lie in $U$; see \cite[Figure~2]{walter}.

By Proposition~\ref{prop:branchedcover}, to prove that $f \colon U \to U$ is vanilla we would need to show that each critical value of $f$ only has one critical preimage. Although this is likely to be the case, it seems quite complicated 
  to prove. 

Instead, we use quasiconformal maps to find a function ``close'' to $f$ with the properties we require. Let the critical points of $f$ be $(c_k)_{k\in\N}$, and choose a strictly decreasing sequence of positive real numbers $(r_k)_{k\in\N}$ tending to zero. For each $k$ consider the balls $B_k \defeq B(f(c_k), r_k)$ and the component 
  $\tilde{B}_k\defeq \comp_{c_k}(f^{-1}(B_k))$. 
   If $r_k$ is chosen sufficiently small, then
    \[ \diam(\tilde{B}_k) < \frac{\dist(c_k, \partial U \cup (\CP(f)\setminus \{c_k\}))}{2}. 
          \]  
 Hence $\tilde{B}_k$ and $B_k$ are in $U$, the $\tilde{B}_k$ are pairwise disjoint, and $f \colon \tilde{B}_k \to B_k$ is proper. 

Let $\phi \colon \C \to \C$ be a quasiconformal map such that $\phi(\D) = \D$, such that $\phi(z) = z$ for $z \in\partial \D$, and such that $\phi(0) = 1/2$. We then define a quasiregular map $G \colon \C \to \C$ by
\[
G(z) = 
\begin{cases}
r_k \phi\left(\frac{f(z)-f(c_k)}{r_k}\right) + f(c_k), \quad & \text{if } z \in \tilde{B}_k \text{ for some } k \in \N, \\
f(z), \quad & \text{otherwise.}
\end{cases}
\]

It is easy to see that $G$ is indeed quasiregular, with the same quasiconstant as $\phi$. Moreover, $G$ has the same critical points as $f$, all of which are simple, and the critical values of $G$ are $f(c_k) + r_k /2$, for $k \in \N$.

Since $G$ is a quasiregular mapping of the plane, it follows by Sto\"{\i}low factorisation that there is a transcendental entire function $g$ and a quasiconformal map $\psi \colon \C \to \C$ such that $g = G \circ \psi$.

It can be seen that $\AV(g) = \{0\}$. The critical points of $g$ are simple, and given by $\psi^{-1}(c_k)$, for $k \in \N$. The critical values of $g$ are equal to $f(c_k) + r_k /2$, for $k \in \N$, and all lie in $U$. Thus the (infinitely many) critical values of $g$ accumulate only at $0$, and $g$ is in the Eremenko-Lyubich class. Moreover, each critical point of $g$ has exactly one critical preimage.

We have shown that $g \colon \psi^{-1}(U) \to U$ is vanilla. By Corollary~\ref{cor:maintopological}, it follows that there are disjoint simply-connected domains $U_1, U_2, \ldots \subset U$ each with connected preimage in $\psi^{-1}(U)$. Since $g^{-1}(U) = \psi^{-1} \circ G^{-1}(U) = \psi^{-1}(U)$, this completes the proof.

It remains to show that each $U_j$ can be assumed to be bounded. That this is possible is easily seen as follows. Let $T>0$ be such that $\CV(g)\subset (0,T]$, and let
   $U^1\subset U$ be a bounded simply-connected domain containing $(0,T]$. It is easy to see that $g^{-1}(U^1)$ is also connected. Then the claim follows by
   applying Corollary~\ref{cor:maintopological} to the vanilla function $g\colon g^{-1}(U^1)\to U^1$.
\end{proof}

%
%
%
%

\section{Proof of Theorem~\ref{thm:AVexample}}
\label{sect:AVexample}

We now turn to our construction of an entire function $f$ having two simply-connected domains with connected  
  preimages, each containing a logarithmic asymptotic value of $f$. (An asymptotic value $a$ is \emph{logarithmic}
  if there is a neighbourhood $U$ of $a$ and a connected component $\tilde{U}$ of $f^{-1}(U)$ such that
   $f\colon \tilde{U}\to U\setminus \{a\}$ is a universal covering map.)
Similarly as in
  Theorem~\ref{thm:maintopological}, we could introduce a topological class of functions to which our methods apply. For 
  definiteness and simplicity, let us instead
   study the explicitly given entire function
\begin{equation}
\label{fdef}
f(z) \defeq \frac{2}{e^{1/4}\sqrt{\pi}} \int_0^z \cosh w \exp(-w^2)\ \deriv w = \frac{1}{2} \left(\erf\left(z+\frac{1}{2}\right)+\erf \left(z-\frac{1}{2}\right)\right).
\end{equation}
Here $\erf$ is the \emph{error function} $\erf(z) \defeq \frac{2}{\sqrt{\pi}}\int_0^z \exp(-w^2)\ \deriv w$; the equality in \eqref{fdef} is obtained by an explicit calculation. We shall use the following properties of $f$.

\begin{lem}[Properties of $f$]
\label{lem:f}
  The function $f$ is an odd entire function, real on the real axis with $f(\R)=(-1,1)$, and satisfying
    \[f(\{ iy\colon y > 0 \}) = f(\{ iy \colon y < 0 \}) = \{ iy \colon y \in \R \}. \]
  The critical points of $f$ are all simple and given by 
   $z_n = i\cdot \frac{2n+1}{2} \pi$, for $n\in\Z$. 
    The critical values $\zeta_n\defeq f(z_n)$ are all purely imaginary, and distinct critical points have distinct critical values. Moreover,
    $\im \zeta_n$ is positive for even $n$, and negative for odd $n$.  

  In addition, there are exactly two asymptotic values, namely $1$ and $-1$, and with $f(x)\to \pm 1$ as $x\to \pm \infty$.
    These are logarithmic asymptotic values and there is exactly one tract over each. That is, let 
     $U$ be a sufficiently small Jordan domain containing $a\in \{-1,1\}$, and let 
     $\tilde{U}$ be the connected component of $f^{-1}(U)$ containing  an infinite piece of the real axis.
     Then $f\colon\tilde{U}\to U\setminus\{a\}$ is a universal covering, while all other components of
      $f^{-1}(U)$ are mapped conformally to  $U$ by $f$. 
\end{lem}
\begin{proof}
  The function $f$ is odd as it is the integral of an even function with $f(0)=0$. Since the integrand is real and positive
    on the real axis, $f$ is real and strictly increasing on  $\R$. Since $\erf(x)\to \pm 1$ as $x\to \pm \infty$, 
   the same is true for $f$. In particular, $f(\R)=(-1,1)$, and $-1$ and $1$ are asymptotic values of $f$. 

   The critical points of $f$ are the zeros of $\cosh$, which are as stated in the lemma; they are simple since
    $\cosh$ has only simple zeros. Furthermore, $f$ is imaginary on the imaginary axis. Indeed, if $t\in\R$, then 
\begin{align}\label{eqn:fimaginary}
  f(it) &= \frac{2}{e^{1/4}\sqrt{\pi}} \int_0^{it} \cosh w \exp(-w^2)\ \deriv w \\
      &= \frac{2i}{e^{1/4}\sqrt{\pi}} \int_0^{t} \cos y \exp (y^2)\ \deriv y 
     \eqdef \frac{2i}{e^{1/4}\sqrt{\pi}} \alpha(t).\notag
\end{align}

 For $n\geq 0$, set  $t_n \defeq \frac{2n+1}{2}\pi$ and
   $\alpha_n \defeq \alpha(t_n)$. Observe that $\alpha_0>0$. 
    Since $\zeta_{-n} = -\zeta_{n-1}$ for all $n$, the following claim implies that 
    all $\zeta_n$ are indeed pairwise distinct, and that 
    $\im \zeta_n$ is positive if and only if $n$ is even. 

\begin{claim}
   The sequence $\lvert \alpha_n\rvert$ is strictly increasing
    and $\alpha_n$ is positive exactly for even $n$. 
\end{claim}
\begin{subproof}
Note that 
$\alpha_{n+1} = \alpha_n + \beta_n$, where 
\[
\beta_n \defeq \int_{\frac{2n+1}{2}\pi}^{\frac{2n+3}{2}\pi} \cos y \exp (y^2) \ \deriv y.
\]
 Note that $\beta_n$ is positive exactly when $n$ is odd. We now estimate $\beta_n$ from below in terms of
    $\alpha_n$. Note that 
    $\lvert \alpha(t)\rvert \leq \exp(t^2)$ for all $t$.  A simple estimate shows that 
\[
\lvert \beta_n \rvert > 
\exp\left(\left(\frac{3n+2}{3}\right)^2\pi^2\right) > 
2 \exp\left(\left(\frac{2n+1}{2}\right)^2\pi^2\right) \geq 2 \lvert \alpha(t_n)\rvert = 2 \lvert \alpha_n\rvert . 
\]
Hence $\lvert \alpha_{n+1}\rvert > \lvert \alpha_n\rvert$ for all $n$, and 
   $\alpha_{n+1}$ has the same sign as $\beta_n$. This proves the claim.
\end{subproof}

    Note that the order of $f$ (see \cite[p.~219, \P~181]{nevanlinnagerman}) is $\rho(f)=2$. By the Denjoy-Carleman-Ahlfors 
        theorem~\cite[p.~313, \P~269]{nevanlinnagerman}, the number $m$ of asymptotic
        values of $f$ is finite. More precisely, 
        $2m_{\direct}+m_{\indirect}\leq 4$, where 
        $m_{\direct}$ and $m_{\indirect}$ are the numbers
        of \emph{direct} and \emph{indirect} singularities of $f^{-1}$, respectively,
        over finite asymptotic values.
        (Compare~\cite[p.~289, \P~245]{nevanlinnagerman},  \cite{bergweilereremenkosingularities} or Section~\ref{sect:itdoeshold} for definitions.)
        As the 
        set of critical values of $f$ is discrete, every asymptotic value $a$ of $f$ 
        is an isolated point of $S(f)$. Hence any singularity over $a$ is logarithmic, and 
        therefore direct. So $m=m_{\direct}\leq 2$, and there are no asymptotic
        values except $-1$ and $1$, as claimed. 
\end{proof}


In summary, $f$ has a structure similar to a vanilla function, apart from the two asymptotic values. 
Let $\CV^+(f)=\{\zeta_{n}\colon n\geq 0\}$ denote the set of critical values corresponding to critical points with positive imaginary part, and let $\CV^-(f)=\CV(f)\setminus \CV^+(f)$ denote the set of critical values corresponding to critical points with negative imaginary part.

 Recall that we defined the notion of partial configurations only  for vanilla functions; we shall next
   introduce a version of this notion specific to our function $f$.
\begin{defn}[Partial configuration]
\label{def:partialconfasymptotic}
  A triple $(D,U,V)$ of Jordan domains
  is called a \emph{partial configuration} (for the function $f$ from~\eqref{fdef}) if all the following hold.
  \begin{enumerate}[(a)]
      \item $[-1,1]\subset D$.
      \item $-1 \in U$ and $1 \in V$.
      \item $\overline{U}\cap \overline{V}=\emptyset$ and $\overline{U} \cup \overline{V} \subset D$.
      \item $D\cap \CV(f)\subset U\cup V$ and $\partial D\cap S(f)=\emptyset$. 
      \item The component $\tilde{D}\defeq \comp_{\R}(f^{-1}(D))$ satisfies
          $f(\tilde{D}\cap \CP(f))=D\cap \CV(f)$.
      \item Let $\tilde{U}$ and $\tilde{V}$ be the components of $f^{-1}(U)$ and $f^{-1}(V)$, respectively, that
         contain an infinite piece of the real axis. Then 
            $\tilde{D}\cap \CP(f) \subset \tilde{U}\cup\tilde{V}$. 
   \end{enumerate}
\end{defn}

  The mapping properties of $f$ on a partial configuration are 
     illustrated in Figure~\ref{fig:1asymptotic} and in the following analogue of Proposition~\ref{prop:picture}.
\begin{figure}
	\includegraphics[width=\textwidth]{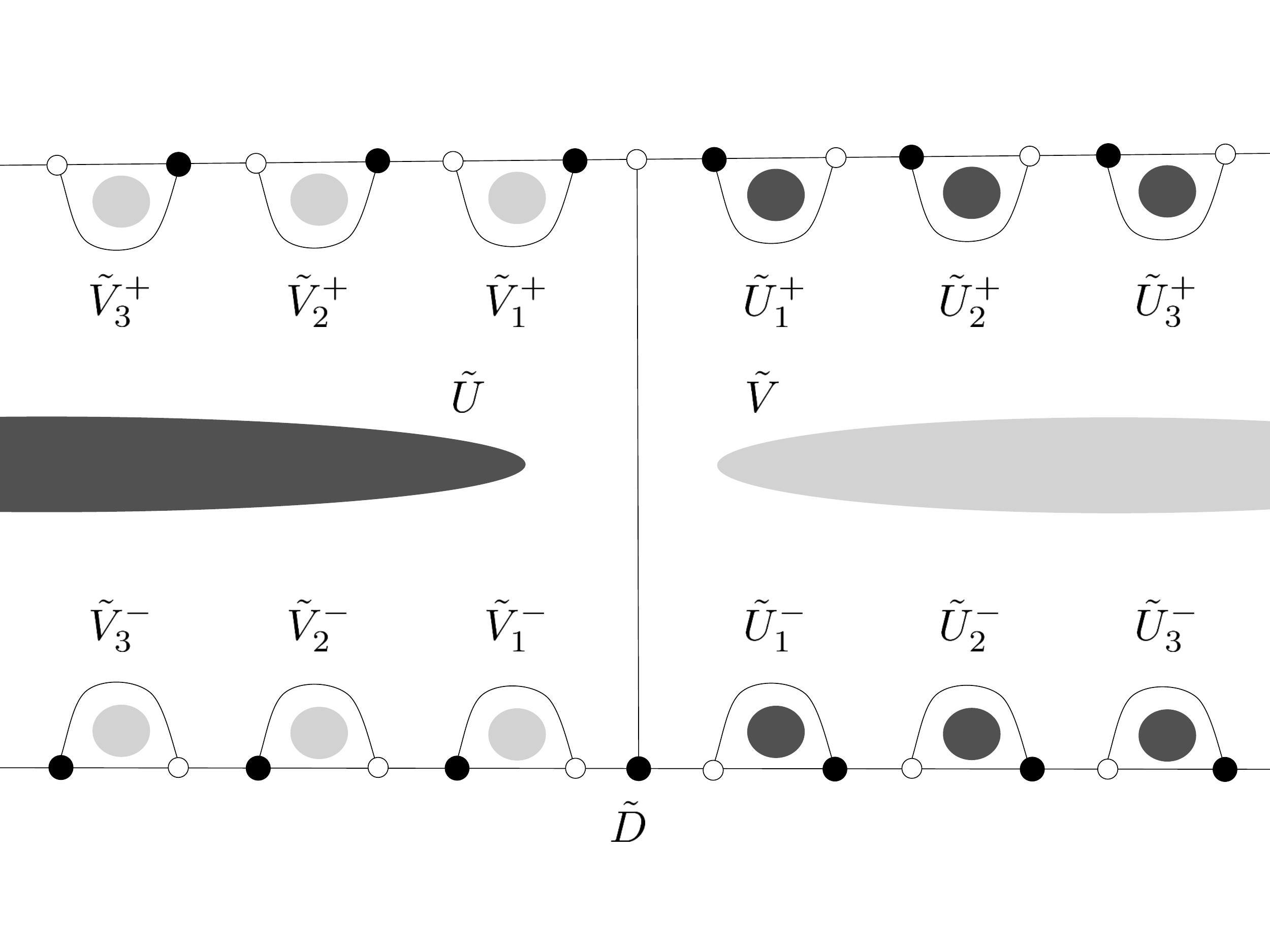}
  \caption{An illustration of Proposition~\ref{prop:pictureasymptotic}. 
    The structure of
    the partial configuration $(D,U,V)$ is the same as shown in the right-hand side of Figure~\ref{fig:1}, so we do 
    not repeat it here. Shown is the preimage component $\tilde{D}$ of $D$, where 
   shades of grey indicate preimages of $U$ and $V$, and white and black circles indicate preimages of $w$ and $z$ respectively. We stress again that this picture is topological rather than geometrically accurate. We also remind the reader that $D$ has other preimage components under $f$; these are not illustrated.}\label{fig:1asymptotic}
\end{figure}

\begin{prop}[Structure of partial configurations]
\label{prop:pictureasymptotic}
Suppose that $(D,U,V)$ is a partial configuration, and let $\gamma$ be a crosscut of $D$ that
   separates $\overline{U}$ from $\overline{V}$. 
Then:
  \begin{enumerate}[(a)]
    \item \label{item:boundary}%
       The boundary  $\partial \tilde{D}$ consists of two injective curve tending to infinity in both directions: one in the 
       upper half-plane (the \emph{upper boundary} $\partial^+\tilde{D}$) 
       and one in the lower half-plane (the \emph{lower boundary} $\partial^-\tilde{D}$). 
    \item Each component of $\tilde{D}\setminus f^{-1}(\gamma)$ (which we shall call a \emph{face}) 
          contains either exactly one component of $f^{-1}(U)$ or exactly one component of $f^{-1}(V)$. 
           A \emph{simple} face is one that is mapped conformally by $f$. 
    \item The faces containing $\tilde{U}$ and $\tilde{V}$ are adjacent to each other, separated by 
      a preimage $\tilde{\gamma}$ of $f^{-1}(\gamma)$ that connects the upper and lower boundaries of $\tilde{D}$. 
    \item Every component of $f^{-1}(U)\cap\tilde{D}$ apart from $\tilde{U}$ is contained in a simple face 
          adjacent to the one containing $\tilde{V}$. Similarly, each component of $f^{-1}(V)\cap\tilde{D}$ 
          apart from $\tilde{V}$ is 
          contained in a simple face 
          adjacent to the one containing $\tilde{U}$. 

    The components of $(f^{-1}(U)\setminus \tilde{U})\cap \tilde{D}$ contained in a face whose boundary intersects the lower boundary of 
     $\tilde{D}$ will be labelled $(\tilde{U}^-_j)_{j=1}^{\infty}$, with $\tilde{U}_1^-$ closest to $\tilde{\gamma}$ and 
     proceeding in positive orientation. We similarly define $(\tilde{U}_j^+)$, $(\tilde{V}_j^-)$, and $(\tilde{V}_j^+)$ (see Figure~\ref{fig:1asymptotic}). 
    \end{enumerate}
\end{prop}
\begin{proof}
  Recall that all singular values of $f$ in $D$ lie in $U$ or $V$. Let $D_L$ be the component of $D\setminus \gamma$
     containing $U$. Then $A\defeq D_L\setminus \overline{U}$ is an annulus containing no singular values. 
     Hence $f$ is a covering map on every component of $f^{-1}(A)$. Thus each such component is either
     simply-connected and mapped as a universal covering, or doubly-connected and mapped as a finite
     covering of annuli. In particular, every component of $f^{-1}(A)$ is adjacent to exactly one component of
     $f^{-1}(U)$, and every component of $f^{-1}(D_L)$ contains exactly one component of
     $f^{-1}(U)$. 

   Recall that any component of $f^{-1}(U)$ different from $\tilde{U}$ is mapped conformally by $f$, and hence
    contained in a simple face that is mapped conformally to $D_L$. The analogous statement holds for the preimages
    of $D_R$, so we see that there are exactly two non-simple faces, namely the face $\tilde{D}_L$ containing
    $\tilde{U}$ and the face $\tilde{D}_R$ 
    containing $\tilde{V}$. Since no simple face is adjacent to more than one other face, and $\tilde{D}$ is connected,
    we see that $\tilde{D}_L$ and $\tilde{D}_R$ are adjacent, and that every simple face is 
    adjacent to one of these two. 
    
  Let $\tilde{\gamma}$ be the component of $f^{-1}(\gamma)$ that separates $\tilde{U}$ and $\tilde{V}$ in $\tilde{D}$. 
    Recall that $\R\subset\tilde{D}$, and that $\tilde{U}$ (resp. $\tilde{V}$) contain an infinite piece of the negative (resp. 
    positive) real axis. It follows that $\tilde{\gamma}$ must have one endpoint in the upper half-plane and
    one endpoint in the lower half-plane. 
 
  It remains to establish~\ref{item:boundary}.     
    Since $\partial D$ contains no singular values and $\tilde{D}$ is unbounded, every component of  $\partial\tilde{D}$
     is an injective curve tending to infinity in both directions, and $f$ maps this component to
     $\partial D$ as a universal covering. We must show that there are only two such components.
  To do so, we claim that 
   there is exactly one component of $f^{-1}(A)$ adjacent to $\tilde{U}$ (and similarly for  $\tilde{V}$).
   In other words, $\partial \tilde{U}$ and $\partial \tilde{V}$ are connected. 

   Let $U'$ be obtained from $U$ by removing, for each critical value $c\in U$, an arc $\delta_c$ connecting
     $c$ to $\partial U$, in such a way that these arcs are pairwise disjoint. Let $\tilde{U}'$ be 
     the connected component of $f^{-1}(U')$ containing an infinite piece of the real axis.
     By Lemma~\ref{lem:f}, $f\colon \tilde{U}'\to U$ is a universal covering, and every other component of
      $f^{-1}(U')$ is mapped conformally. Since $\tilde{U}$ contains only finitely many critical points,
      it follows that there are only  finitely many preimage components of $U'$ in $\tilde{U}$. The claim follows
      easily, as does the fact that $\partial\tilde{D}$ has only two connected components.
\end{proof}

 The central part of our construction is a method for extending partial configurations, in the same spirit as
    Proposition~\ref{prop:inductivestep}. 
\begin{prop}[Extending partial configurations]
\label{prop:nextdomain}
Suppose that $(D,U,V)$ is a partial configuration, and let
    $P$ be a connected component of $f^{-1}(U)\cap \tilde{D}$. Then there is
    a partial configuration $(D', U', V)$ such that 
$D \subset D'$, $U \subset U'$ and $P\subset \tilde{U}'$.

   Moreover, for any $R>0$, this configuration can be chosen such that $B(0,R)\subset {D}'$.
\end{prop}
\begin{proof}
  We first prove the proposition without the final statement. 
  We may assume that $\tilde{P}\neq \tilde{U}$ (otherwise, there is nothing to prove). 
   Let $\gamma$ be as in 
     Proposition~\ref{prop:pictureasymptotic} and recall that $P$ is contained in a simple face adjacent to 
    the face $\tilde{D}_R$ containing $\tilde{V}$. The boundary of this simple face consists of a preimage of
    $\gamma$ and a piece of $\partial \tilde{D}$; we may assume without loss of generality that this 
    piece belongs to the lower boundary of $\tilde{D}$. (Otherwise, since $f(\overline{z}) = \overline{f(z)}$, for $z \in \C$, we can replace $D$,  $U$ and $V$ by their 
    reflections in the real axis.)
   
   Let $p\geq 1$ be such that the $P=\tilde{U}^-_p$.  (In Figure~\ref{fig:10}, $p=2$.) 
   We now extend $D$ to a Jordan domain $D'\supset \overline{D}$ such that:
\begin{itemize}
\item $\partial D' \cap \CV(f) = \emptyset$;
\item $\#((D' \setminus D) \cap CV^+(f)) = p$;
\item $\#((D' \setminus D) \cap CV^-(f)) = 0$;
\item $f(\tilde{D'} \cap \CP(f)) = D' \cap \CV(f)$, where $\tilde{D}=\comp_{\R}(f^{-1}(D'))$. 
\end{itemize}
 This can be  achieved by the same technique as in the proof of Proposition~\ref{prop:Dn}:
    First extend $D$ to $D_1$ such that $\tilde{D}_1\defeq \comp_{\R}(f^{-1}(D_1))$ contains at least $p$
    points of $\CP(f)\setminus \tilde{D}$ whose images have 
    positive imaginary part. Then remove all critical values whose critical preimages are not
    in $\tilde{D}_1$ as in Proposition~\ref{prop:Dn}. We can also remove arcs connecting the points of $\CV^-(f)\cap D_1$ to $\partial D_1$ without intersecting $D$.  
    Observe that the critical preimage of any such arc lies in the lower half-plane, and hence does not separate $\tilde{D}\supset\R$ from any 
    of the critical points in the upper half-plane. Finally, remove any excess critical values in $\CV^+(f)$ exactly as in Proposition~\ref{prop:Dn}.

It now remains to extend $U'$ to $U$. 
   Let $w$ be the endpoint of ${\gamma}$ in the upper half-plane, and $z$ its endpoint in the lower half-plane.
 As in the proof of Proposition~\ref{prop:inductivestep}, we construct an arc $\alpha$ 
  joining $w$ to $\partial D'$ so that all the critical values in $D'\setminus D$ lie on $\alpha$, and such that one connected component $\tilde{\alpha}$ of
  $f^{-1}(\alpha)$ contains all points of $\CP(f)\cap \tilde{D}' \setminus \tilde{D}$. Then we let $\beta$ 
  be the sub-arc of $\alpha$ starting at $w$ and ending at the last critical value on $\alpha$.

    Consider an arc $\Gamma$, not intersecting $\overline{V} \cup \CV(f)$ and not intersecting  $\overline{U}$ except in
      one endpoint,  defined as follows. The arc starts at some point $\mu\in\partial U$, 
      and runs around $\beta$ and $\overline{V}$ exactly $p$ times in negative orientation.
      On its last loop it enters $D$ at
      the point $z$, traversing along the arc $\gamma$ and ending at $w$. 

       \begin{figure}
      	\includegraphics[width=\textwidth]{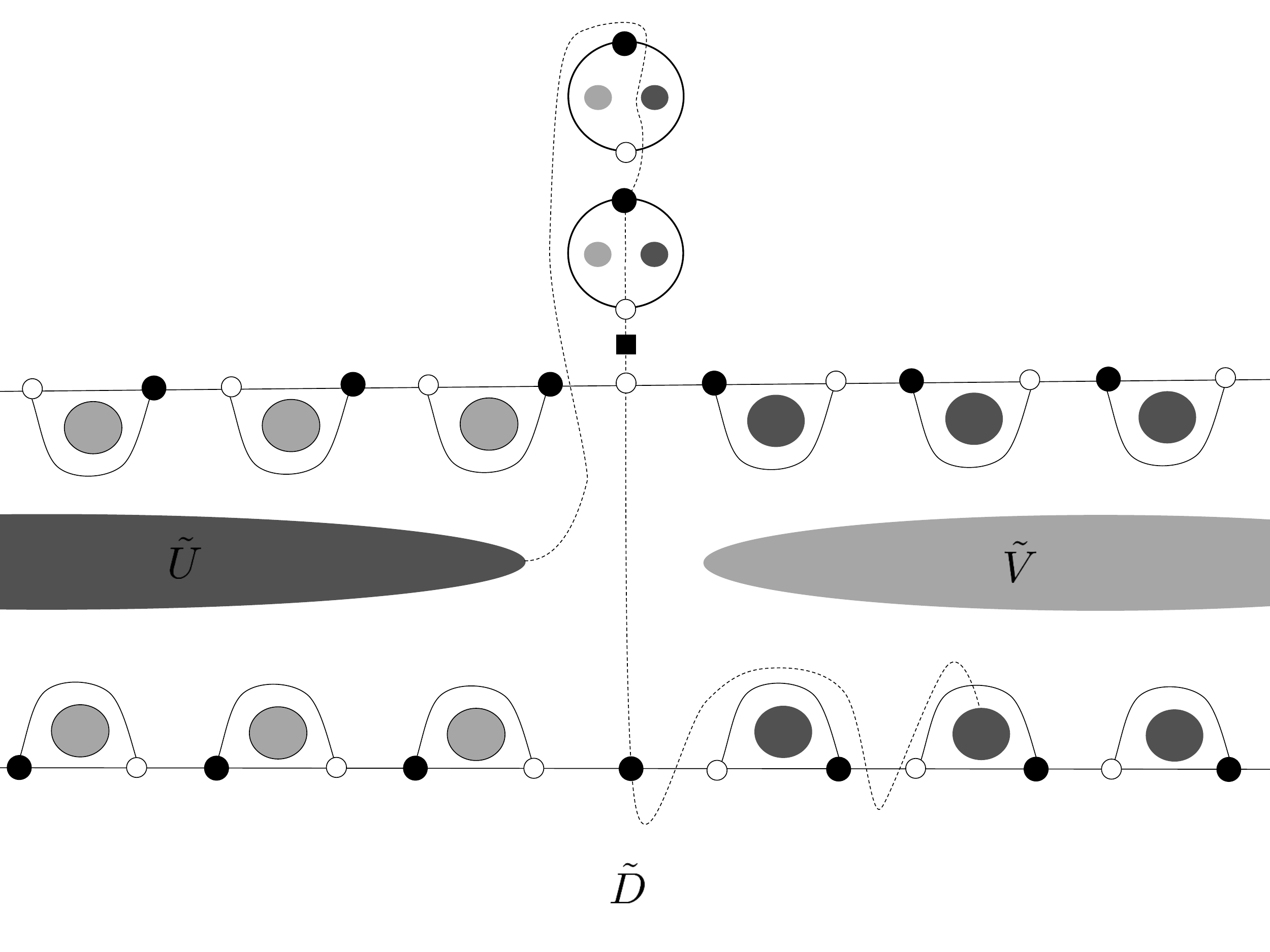}
        \caption{The extended preimage in Proposition~\ref{prop:pictureasymptotic}, in the case where $p=2$. 
        The image configuration  is the same as the right-hand side of Figure~\ref{fig:5}. 
        The preimage of the thin channel $W$ starts at $\tilde{U}$, passes through two preimages of $D$, meets a critical point (shown as a square), and ends on $\tilde{U}^-_2$. Note that other preimages of critical values, shown in Figure~\ref{fig:5}, are not shown here, for simplicity.}\label{fig:10}
      \end{figure}

         Take a thin channel $W$ containing the Jordan arc $\Gamma \cup \beta$, in such
          a way that $U' \defeq U \cup W$ is simply-connected. 
         By an argument similar to that in the proof of Proposition~\ref{prop:inductivestep}
         it can be shown that $\tilde{U}'$ contains both $\tilde{U}$ and $\tilde{U}^-_p$; see Figure~\ref{fig:10}.
         This completes the construction.

  Now let us prove the final claim of the proposition. As in Proposition~\ref{prop:Dn}, we can find $D''\supset D'$ such that 
    $\tilde{D}''\defeq \comp_{\R}(f^{-1}(D''))$ contains all critical preimages of points in $\CV(f)\cap D''$, and such that furthermore
    $B(0,R)\subset {D}''$. We can find an arc $\alpha^+\subset D''\setminus D'$, having one endpoint on $\partial D'$ and another in $D''\setminus D'$, and 
    passing through all points of $(D'' \setminus D')\cap \CV^+(f)$, with the following property: there is a connected component of $f^{-1}(\alpha)$ that 
    intersects $\partial{\tilde{D}'}$ and contains all critical preimages of the critical values in $\alpha^+$. We can find a similar curve $\alpha^-$ (disjoint from $\alpha^+$) for 
   the critical values in $(D'' \setminus D')\cap \CV^-(f)$. By extending $U'$ appropriately to include these two arcs, we easily obtain $U''\supset U'$ such that 
    $\tilde{U}''\defeq \comp_{\tilde{U}'}(f^{-1}(U''))$ contains $\CP(f)\cap (\tilde{D}''\setminus \tilde{D}')$. We leave the details to the reader.  
\end{proof}
\begin{rmk}[Symmetry]\label{rmk:symmetry}
  Recall that $f$ is odd. Hence, if $(D,U,V)$ is a partial configuration for $f$, then so is 
    $(-D,-V,-U)$. It follows that Proposition~\ref{prop:nextdomain} holds also with the roles of $U$ and $V$ exchanged. 
\end{rmk}
Applying the preceding step inductively, we obtain the following fact, which
   easily implies Theorem~\ref{thm:AVexample}.
\begin{prop}[Sequence of configurations]
\label{prop:mainconstructionasymptotic}
  There is a sequence 
     $(D_k, U_k , V_k)_{k=1}^{\infty}$ of partial configuration{s} such that 
    $(D_k)$, $(U_k)$ and $(V_k)$ are increasing sequences, and such that every component of $f^{-1}(U_1)$ is contained in
     $\tilde{U_k}$ for sufficiently large $k$, and similarly for the components of $f^{-1}(V_1)$.
\end{prop}
\begin{proof}
   Let $U_1$ be a small disc containing $-1$, let $V_1$ be a small disc containing $1$, and let $D_1$ be a 
   simply-connected neighbourhood of $[-1,1]$ containing $U_1 \cup V_1$ and such that $D_1 \cap \CV(f) = \emptyset$. 
   It is easy to see that $(D_1, U_1, V_1)$ is a partial configuration. 

 Let $(k(\ell),j(\ell),\sigma(\ell))_{\ell=1}^{\infty}$ be an enumeration of 
    the countable set $\N\times\N\times \{-,+\}$. We may assume that
    $k(\ell)\leq \ell$ for all $\ell$. We now describe how we construct $(D_k,U_k,V_k)$ inductively
   from $(D_{k-1}, U_{k-1}, V_{k-1})$, for $k\geq 2$. 

 First suppose that $k$ is even, say $k=2\ell$ with $\ell\geq 1$. Set $V_k\defeq V_{k-1}$. 
   Let $P$ be the connected component of 
   $f^{-1}(U_{k-1})$ that contains $(\tilde{U}_{k(\ell)})_{j(\ell)}^{\sigma(\ell)}$. 
   (Recall the notation for the connected components of $f^{-1}(U_{k(\ell)})\cap \tilde{D}_{k(\ell)}$ from
    Proposition~\ref{prop:pictureasymptotic}.) 
    We apply Proposition~\ref{prop:nextdomain} to obtain a partial configuration $(D_k,U_k,V_k)$ such that 
     $B(0,k)\subset {D}_k$ and 
   \[ (\tilde{U}_{k(\ell)})_{j(\ell)}^{\sigma(\ell)} \subset \tilde{U}_k. \]

 Now suppose that $k$ is odd, say $k=2\ell+1$. Using Remark~\ref{rmk:symmetry}, we proceed exactly as in 
    the previous step, but with the roles of $U$ and $V$ exchanged. So we obtain 
   a partial configuration $(D_k,U_k,V_k)$ such that 
    $B(0,k)\subset D_k$ and 
   \[ (\tilde{V}_{k(\ell)})_{j(\ell)}^{\sigma(\ell)} \subset \tilde{V}_k. \]

 This completes the inductive construction. Let $P$ be any connected component of 
   $f^{-1}(U_1)$. We must show that $P\subset \tilde{U}_k$ for sufficiently large $k$. 

By construction, $P\cap \tilde{D}_{k_0}\neq\emptyset$ for some $k_0$.
    Since $P\cap \partial \tilde{D}_{k_0}=\emptyset$, in fact $P\subset \tilde{D}_{k_0}$.     
    If $P \subset \tilde{U}_{k_0}$, then we are done. Otherwise, $P\subset (\tilde{U}_{k_0})_j^{\sigma}$ for some $j\geq 1$ and $\sigma\in \{-,+\}$. Let 
    $\ell$ be such that $k(\ell)=k_0$, $j(\ell)=j$ and $\sigma(\ell)=\sigma$; recall that $\ell\geq k_0$. So, by construction,
    \[ P \subset (\tilde{U}_{k(\ell)})_{j(\ell)}^{\sigma(\ell)} \subset \tilde{U}_{2\ell} \subset \tilde{U}_k \]
    for $k\geq 2\ell$. The same argument applies to preimage components of $f^{-1}(V_1)$, and the proof is complete. 
\end{proof} 

\begin{proof}[Proof of Theorem~\ref{thm:AVexample}]
   Set \[U\defeq\bigcup_{k\in\N} U_k\quad\text{and}\quad V\defeq \bigcup_{k\in\N} V_k.\] 
   The result follows in exactly the same
    way as in the proof of Theorem~\ref{thm:maintopological}.
\end{proof}

\begin{rmk}[More asymptotic values]
  For any  $2\leq d \leq\infty$, 
    it should be possible to use a similar construction to obtain an entire function  $f$ having $d$ logarithmic asymptotic 
    values $(a_j)_{j=1}^{d}$, and pairwise disjoint simply-connected domains $(U_j)_{j=1}^d$ with  $a_j\in U_j$, such
    that $f^{-1}(U_j)$ is connected for all~$j$.
\end{rmk}

%
%
%
%
\section{The error in Baker's proof}
\label{sect:bakerflaw}

       \begin{figure}
      	\includegraphics[width=\textwidth]{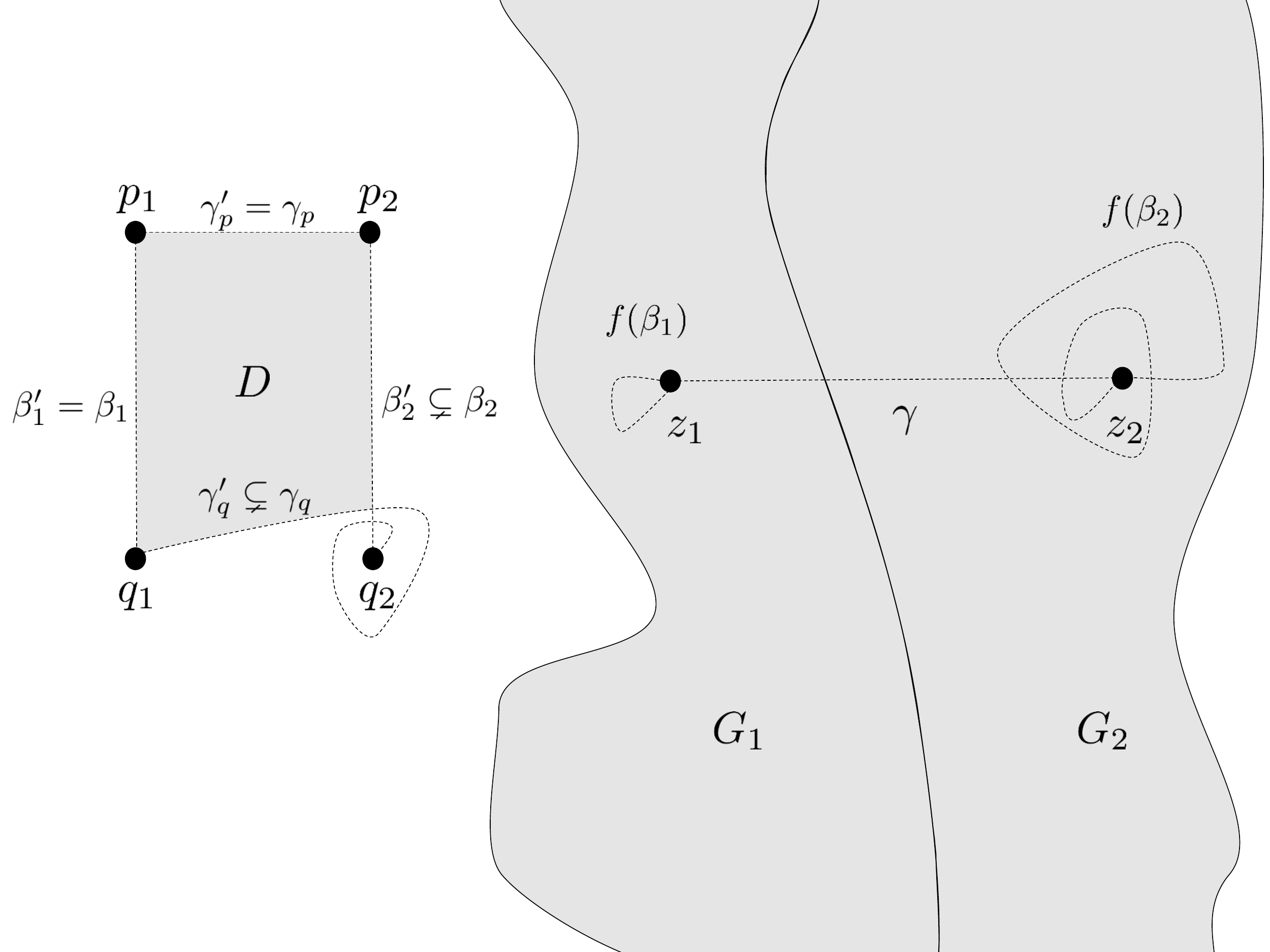}
        \caption{This figure
         displays the idea in the proof in \cite{baker1970}. 
         The left-hand side shows the domain of the function, and the right-hand side shows the image. 
         The domains $D, G_1$ and $G_2$ are shaded. The images $f(\beta_1)$ and $f(\beta_2)$ are shown as dotted lines.}\label{fig:7}
      \end{figure}

In this section we briefly outline the proof in \cite{baker1970}, and highlight where the error occurs. As mentioned earlier, the proof in \cite{baker1970} amounts to a positive answer to Question~\ref{ques:main}, so we suppose that $f$ is a transcendental entire function, and that $G_1$ and $G_2$ are disjoint simply-connected domains each with connected preimage. Baker attempts to deduce a contradiction from this.

Baker uses a well-known result, known as the Gross star theorem~\cite[p.~292, \P~247]{nevanlinnagerman}, which we also use later.
\begin{theorem}[Gross star theorem]
\label{theo:gross}
Suppose that $f$ is a transcendental entire function and that $\phi$ is a holomorphic branch of the inverse of $f$ defined in a neighbourhood of a point $w$. Then, for almost all $\theta\in [0,2\pi)$, the branch $\phi$ can be continued analytically along the ray
   $\{w+ t\cdot e^{i\theta}\colon t\geq 0\}$.
\end{theorem}

Baker begins by choosing a point $z_1 \in G_1$, with simple preimages $p_1$ and $q_1$. Let 
   $\phi_p$ and $\phi_q$ be the branches of $f^{-1}$ taking $z_1$ to $p_1$ and $q_1$, respectively. 
   By Theorem~\ref{theo:gross}, 
    there is a line segment $\gamma$ joining $z_1$ to a point $z_2 \in G_2$ such that 
    both branches can be continued analytically along $\gamma$. Then
    $\gamma_p\defeq \phi_p(\gamma)$ is an arc joining $p_1$ to a point $p_2\in f^{-1}(z_2)$, and likewise
    $\gamma_q \defeq \phi_q(\gamma)$ joins $q_1$ to a different preimage $q_2\in f^{-1}(z_2)$.
     For each $k \in \{1, 2\}$, we know that $f^{-1}(G_k)$ is connected by assumption, and so we can let 
     $\beta_k \subset f^{-1}(G_k)$ be an arc joining $p_k$ and $q_k$.

The curves $\beta_1$ and $\beta_2$ may intersect the arcs $\gamma_p$ and $\gamma_q$ in some
   interior points.
   Baker notes that by taking suitable subcurves, $\gamma_p'$ of $\gamma_p$, $\gamma_q'$ of $\gamma_q$, $\beta_1'$ of $\beta_1$, and $\beta_2'$ of $\beta_2$, there is a bounded quadrilateral
$D$ with boundary $\beta_1' \cup \gamma_p' \cup \beta_2' \cup \gamma_q'$; see the left-hand side of Figure~\ref{fig:7}. 

Since $G_1$ and $G_2$ are simply-connected, the curves $f(\beta_1)$ and  $f(\beta_2)$ cannot 
   surround $z_2$ and $z_1$, respectively. Observe that
\[
\partial f(D) \subset f(\partial D) \subset f(\beta_1) \cup \gamma \cup f(\beta_2).
\]
 Baker's claim is that an entire function
   cannot map a quadrilateral in this manner, since otherwise $f(D)$ must be unbounded, which is impossible for
    bounded 
    $D$; see Figure~\ref{fig:7}.

However, as pointed out by Duval (personal communication), 
   it is topologically quite possible for $f(D)$ to be bounded,
  in the absence of additional assumptions. This can been seen in the right-hand side of Figure~\ref{fig:6}.

       \begin{figure}
      	\includegraphics[width=\textwidth]{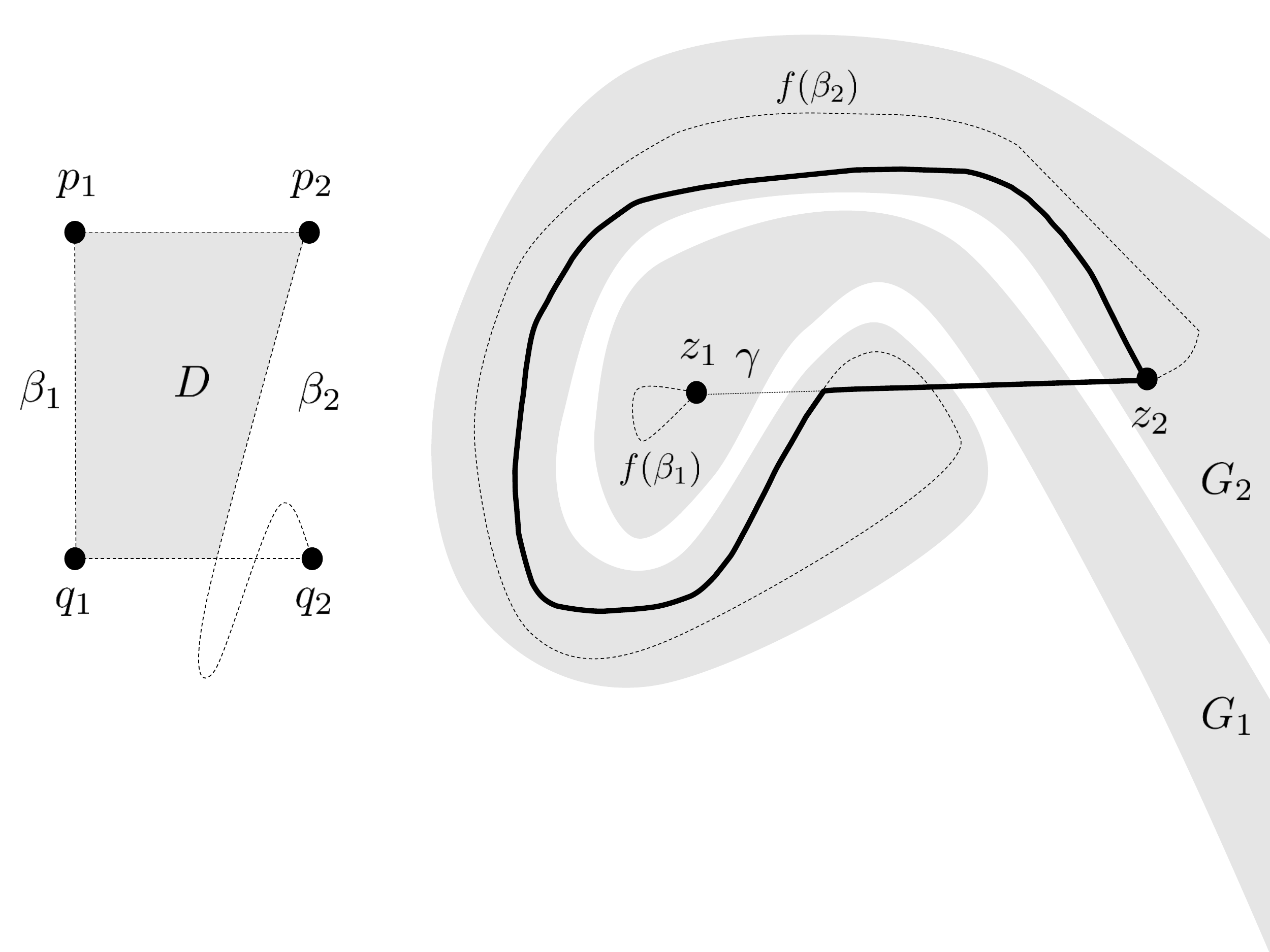}
        \caption{This figure, which is based on one first drawn by Duval,
         displays the error in the proof in \cite{baker1970}. 
         The left-hand side shows the domain of the function, and the right-hand side shows the image. 
         The domains $D, G_1$ and $G_2$ are shaded. The images $f(\beta_1)$ and $f(\beta_2)$ are shown as dotted lines. 
         The boundary of the bounded region $f(D)$ is shown in solid, and is made up of part of $\gamma$ together with 
         part of $f(\beta_2)$.}\label{fig:6}\label{fig:duval}
      \end{figure}

It is interesting to note a key feature of Figure~\ref{fig:6}: although $G_2$ is simply-connected, it
  still ``loops around'' $z_1$ sufficiently far as to intersect $\gamma$ again. It is this ``looping'' which leads to the boundedness of $f(D)$. By examining Figure~\ref{fig:6}, it can be seen this would not occur if we chose the point $z_2 \in \gamma \cap G_2$ in the other component of $\gamma \cap G_2$. This explains how the intricate ``looping'' behaviour arises in the proof in Proposition~\ref{prop:inductivestep}.
  Indeed, consider the case of that proposition where $n=1$ and $m_V=1$. In this case, our construction
  leads exactly to the structure from Figure~\ref{fig:duval}.
  
	\begin{figure}
      	\includegraphics[width=\textwidth]{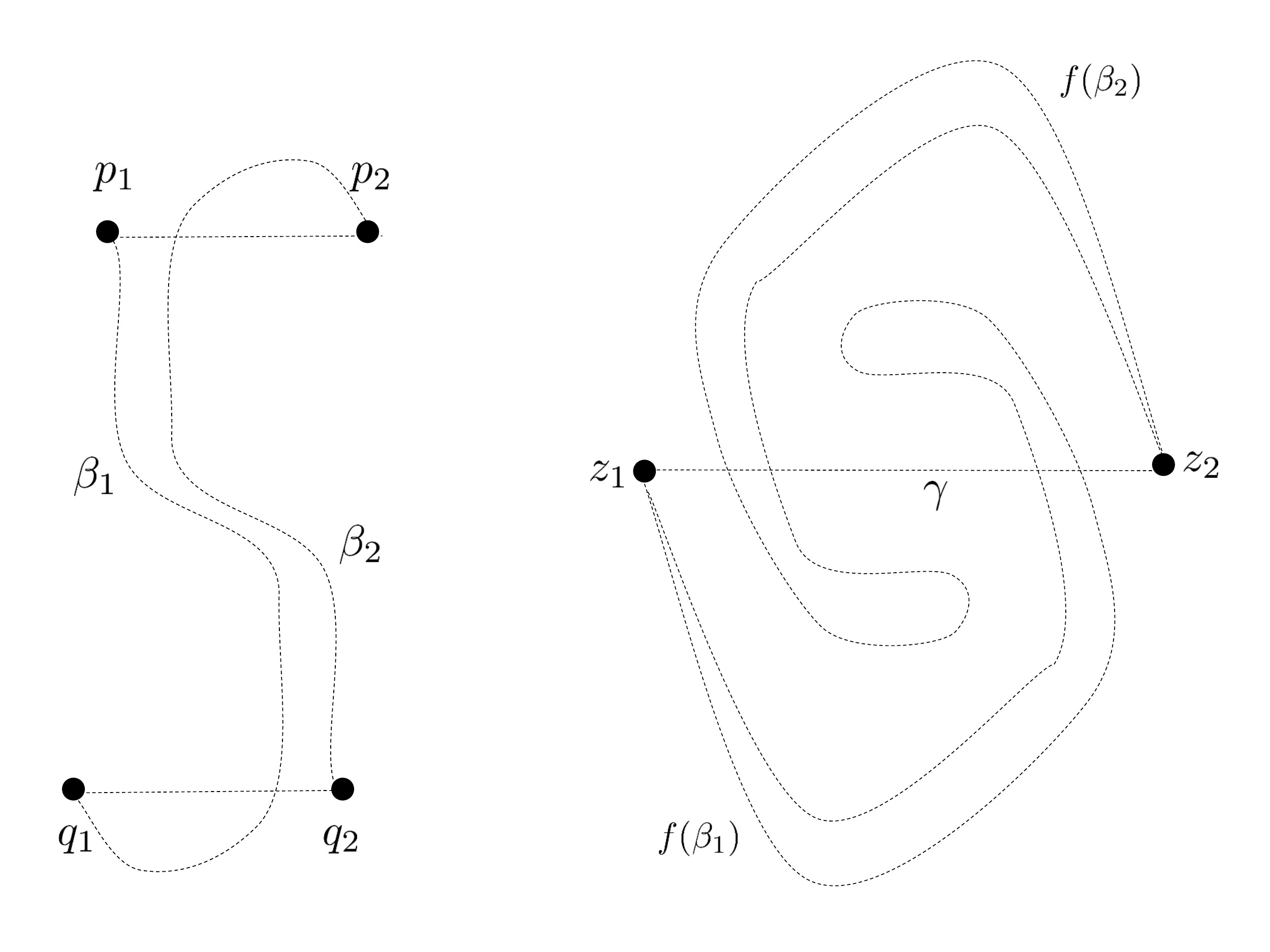}
        \caption{This figure, which is complementary to Figure~\ref{fig:duval}, 
           illustrates a second configuration that leads to Baker's original proof 
              breaking down.}\label{fig:case2}
      \end{figure}
			
We remark that there is another way in which Baker's argument can fail, 
   where instead of one of the domains looping entirely around they other, both
   partly loop around the curve $\gamma$; see Figure~\ref{fig:case2}. 
   The following proposition shows that 
    Baker's method of proof does apply whenever neither kind of looping
    occurs. We shall use it in the next section. 
    
\begin{prop}[Baker's argument]
\label{prop:baker}
Suppose that $f$ is analytic in a simply-connected domain $U$. 
   Suppose that $\gamma_p,\gamma_q\subset U$ are disjoint arcs 
  on both of which $f$ is injective, and with $\gamma \defeq f(\gamma_p) = f(\gamma_q)$. 
  Let $p_1$ and  $p_2$ be the endpoints of $\gamma_p$, and let $q_1$ and  $q_2$ be the 
  endpoints of $\gamma_q$. We may choose the labelling such that $z_1\defeq f(p_1)=f(q_1)$ and 
   $z_2\defeq f(p_2)=f(q_2)$. 

  Let $\beta_1,\beta_2\subset U$ be two arcs 
   such that $\beta_k$ joins $p_k$ and $q_k$, and such that $C_1\defeq f(\beta_1)$ and $C_2\defeq  f(\beta_2)$ 
    are disjoint.

Then every point $\zeta\in \gamma\setminus (C_1\cup C_2)$ 
     is on the boundary of a bounded connected component of 
      $\C\setminus (C_1\cup C_2\cup\gamma)$.
\end{prop}
\begin{proof}
   Let $\zeta\in\gamma\setminus (C_1\cup C_2)$, and let
   $\zeta'$ be the preimage of $\zeta$ on $\gamma_p$. Since
    $\zeta'\notin (\beta_1\cup \beta_2)$, there is a bounded connected component
    $D\subset U$ of $\C\setminus (\gamma_p\cup \gamma_q\cup \beta_1\cup\beta_2)$ with
    $\zeta'\in\partial D$. 

    Then $\partial f(D) \subset f(\partial D) \subset C_1\cup C_2 \cup \gamma$.
Let $\Delta$ be a disc around $\zeta'$ chosen small enough 
       that $f(\Delta)$ does not intersect $C_1\cup C_2$, and let 
       $V$ be the connected component of $f(D\cap \Delta)$ with
       $\zeta\in\partial V$. Since $V\subset f(D)$, it follows that $V$
       is contained in a bounded component of $\C\setminus (C_1\cup C_2\cup \gamma)$, as claimed.
       \end{proof}
       \begin{rmk}[Winding behaviour]
 In order to illustrate how the conclusion of the proposition 
    reflects the ``looping'' behaviour 
    mentioned above (and illustrated in
    Figures~\ref{fig:duval} and~\ref{fig:case2}),
      we remark that it can be reformulated as follows. 
    Suppose that $\zeta\in \gamma\setminus (C_1\cup C_2)$. Then 
    there is $k\in\{1,2\}$ such that either $C_k$ intersects $\gamma$ on
    both sides of $\zeta$, or $C_k\cup\gamma$ contains a Jordan curve
    surrounding $\zeta$.
    
    Indeed, the reformulation clearly implies the conclusion 
      as stated. Conversely, 
        suppose that $\zeta$ is on the boundary of such a bounded component $V$
        of $\C\setminus(C_1\cup C_2\cup \gamma)$.
        We may 
        suppose that $C_1$ and $C_2$ intersect $\gamma$ on different
       sides of $\zeta$, and that neither curve surrounds $\zeta$, as otherwise
        there is nothing to show. Let $\gamma'$ be the connected component of
        $\gamma\setminus (C_1\cup C_2)$ containing $\zeta$; then 
        $\gamma'$ is a cross-cut of the unbounded connected component $W$ of 
        $\C\setminus (C_1\cup C_2)$. 
        
       There must be a second piece $\gamma''$ of $\gamma$ 
        such that $\gamma'\cup \gamma''$ separates $V$ from $\infty$ in $W$. 
         By assumption, there is $k\in\{1,2\}$ such that both endpoints of $\gamma''$ 
           belong to $C_k$. Hence $C_k\cup \gamma''$ separates $\zeta$ from infinity, 
           and the claim follows by choosing a suitable non-intersecting sub-curve.
\end{rmk}

Note that the conclusion of Proposition~\ref{prop:baker} can be 
 strengthened as follows when the curves $\beta_k$ and $\gamma_j$ bound a quadrilateral. 
 (We do not require this fact in the remainder of the paper.)
\begin{rmk}[The case of a quadrilateral]
Suppose that, under the hypotheses of Proposition~\ref{prop:baker}, additionally 
the arcs $\gamma_p$, $\gamma_q$, $\beta_1$ and $\beta_2$ intersect only in their endpoints, and hence bound a quadrilateral $Q$. 
  Then $f(Q)\cup\interior(\gamma)$ is a neighbourhood of $\interior(\gamma)$, and hence 
 no point of $\gamma$ can be connected to infinity without intersecting either $C_1$ or $C_2$. It follows that, in this case, one of the curves $C_1$ and $C_2$ must  surround the other.
\end{rmk}

%
%
%
%
\section{Proof of Theorem~\ref{thm:itdoeshold}}
\label{sect:itdoeshold}
We begin with the following fact, which will be used to deduce Theorem~\ref{thm:itdoeshold}\ref{boundeddomain}.
\begin{prop}[Compact full sets with connected preimages]\label{prop:compactconnected}
Let $f$ be a transcendental entire function, and let $K$ be a 
   compact connected set, containing more than one point, such that $f^{-1}(K)$ is connected. 

    Then $S(f)\cap W = \emptyset$, where $W$ is the unbounded connected component of $\C\setminus K$.
\end{prop}

To prove Proposition~\ref{prop:compactconnected}, 
   we use the following fact: Near any singular value, there
  is an inverse branch $\phi$ and a simple closed curve along which $\phi$ can be continued in such a way as to
  obtain a different inverse branch at the same point. More precisely:
\begin{lem}[Inverse branches near a singular value]\label{lem:inversebranches}
  Let $f$ be a transcendental entire function, and let $s\in\C$. Then the following are equivalent.
    \begin{enumerate}[(a)]
         \item\label{item:singular} $s\in S(f)$; 
       \item\label{item:inversebranches} for every
         neighbourhood $D$ of $s$ there is a polygonal arc
        $\beta\colon [0,1]\to\C$, not passing through any critical points of $f$, such that
            $f(\beta)\subset D$,  $f(\beta(0))=f(\beta(1))$, and such that
          $f\circ \beta$ is injective on $[0,1)$. 
   \end{enumerate}
\end{lem}
\begin{proof}
   Clearly~\ref{item:inversebranches} implies~\ref{item:singular}. 
    So suppose that $s$ is a singular value, and that $D$ is an open disc
    centred at $s$.
  If $D$ contains a critical value $c$ of $f$, then there is a point $\tilde{c} \in \C$ and $d \geq 2$ such that $f$ maps
    like $z \to z^d$ in a neighbourhood of $\tilde{c}$. The claim follows easily.

   So we may suppose that $D$ contains no critical value. Let $\tilde{D}$ be a connected
     component of $f^{-1}(D)$ that is not mapped one-to-one by $f$. Since $D$ contains no critical values,
     it follows that 
     $\tilde{D}$ is unbounded and $f^{-1}(w)\cap \tilde{D}$ is infinite for all $w\in D$, with at most
    one exception. (This follows from a result of Heins \cite[Theorem~4']{heins};
     see \cite[Proposition~2.8]{walternurialasse}.) In particular, there is $w\in D$ having two different
     (simple) preimages $\zeta_1,\zeta_2\in \tilde{D}$. Let $\beta' \subset \tilde{D}$ be a polygonal arc from 
    $\zeta_1$ to $\zeta_2$. Then $f(\beta')$ is a closed curve, but may be self-intersecting. However, 
    since $f$ is locally injective on $\beta'$, it follows that $\beta'$ contains a sub-arc $\beta$ with the required properties.
\end{proof}

 We also need to observe that increasing a set with connected preimage to a larger domain 
   does not change this property.
 \begin{lem}[Increasing sets with connected preimage]\label{lem:increasingsets}
  Suppose that $A\subset\C$ is a connected set containing more than one point, 
     such that $f^{-1}(A)$ is connected. Then $f^{-1}(G)$ is connected for every domain $G$ containing $A$. 
 \end{lem}
\begin{proof}
  Let $U$ be a connected component of $f^{-1}(G)$. Again using \cite[Theorem~4']{heins}, 
   every point in $G$, with at most one exception, has at least one preimage in $U$. In particular,
   $U$ intersects, and hence contains $f^{-1}(A)$. Thus there can be at most one such component. 
\end{proof}

\begin{proof}[Proof of Proposition~\ref{prop:compactconnected}]
  Suppose, by way of contradiction, that $S(f)\cap W \neq\emptyset$. By the Riemann mapping theorem,
    applied to $W\cup\{\infty\}$, there is a Jordan domain $G\supset K$, bounded by an analytic curve $\alpha$,
    such that $S(f)\not\subset \overline{G}$. By Lemma~\ref{lem:increasingsets}, $f^{-1}(G)$ is connected.

  Let $B(w,\eps)\subset \C\setminus \overline{G}$ be a ball intersecting $S(f)$. 
    By Lemma~\ref{lem:inversebranches}, there is a polygonal arc $\beta_1$, joining points $p_1,q_1\in\C$
    with $z_1\defeq f(p_1)=f(q_1)$, such that $f$ is injective on $\beta_1\setminus\{q_1\}$, and such that 
   $C_1 \defeq f(\beta_1) \subset B(w, \epsilon)$ is a Jordan curve through  $z_1$. 

Using the Gross star theorem~\ref{theo:gross}, we construct an arc $\gamma$ that joins $z_1$ to a point 
  $z_2 \in G$, along which the inverse branches $\phi_p$ and $\phi_q$ of $f$ that map $z_1$ to $p_1$ and $q_1$,
  respectively, can both be
  analytically continued. 
  We can choose $\gamma$ such that 
   $\gamma\cap C_1 = \{z_1\}$ and such that
	$\gamma \cap\alpha$ consists of a single point, $\zeta$.

We now have two preimages of $z_2$, corresponding to the analytic continuation of 
   $\phi_p$ and of $\phi_q$, which we label $p_2$ and $q_2$. These are joined to $p_1$ and $q_1$,  
  respectively, by disjoint preimage components $\gamma_p$ and  $\gamma_q$ of $\gamma$.
   Also join $p_2$ and $q_2$ by an arc $\beta_2\subset f^{-1}(G)$.

 By construction, the piece of $\gamma$
    connecting $z_1$ to $\zeta$ is contained in the 
      unbounded connected component of 
      $\C\setminus C_1\cup \alpha$. As $\alpha$ surrounds $C_2$ by definition,
      no point on this arc
      is on the boundary of a bounded connected component of
      $\C\setminus (C_1\cup C_2\cup \gamma)$. This contradicts 
       Proposition~\ref{prop:baker}.
\end{proof}

We next show how to deduce Theorem~\ref{thm:itdoeshold}\ref{compactlyc} from
   Proposition~\ref{prop:compactconnected}, using the following simple fact.

\begin{proposition}[Shrinking a simply-connected domain]
\label{prop:smaller}
Suppose that $f$ is a transcendental entire function, and that $D$ is a simply-connected domain such that 
   $S(f)\cap D$ is compact, and let $\tilde{D}$ be a connected component of $f^{-1}(D)$. 

   If $U\subset D$ is simply-connected with 
   $S(f)\cap  U = S(f)\cap D$, then $f^{-1}(U)\cap \tilde{D}$ is connected.
\end{proposition}
\begin{proof} This is shown in the second paragraph of \cite[Proof of Proposition~2.9]{walternurialasse}.
\end{proof}

We thus have the following strengthening of Theorem~\ref{thm:itdoeshold}\ref{compactlyc}.

\begin{corollary}[Domains with compact intersection with $S(f)$]
\label{cor:compactlyc}
Suppose that $f$ is a transcendental entire function, and that $G$ is a simply-connected domain such that
   $G \cap S(f)$ is compact. Then $f^{-1}(G)$ is connected if and only if $S(f) \subset G$.
\end{corollary}
\begin{proof}[Proof of Corollary~\ref{cor:compactlyc}]
  The ``if'' direction is immediate from Proposition~\ref{prop:smaller}, taking $D=\C$ 
    and $U=G$. (See also \cite[Proposition~2.9~(2)]{walternurialasse}.) 

   For the ``only if'' direction, suppose that  $f^{-1}(G)$ is connected. Let $D\subset G$ be a bounded Jordan domain with
    $S(f)\cap G \subset D$ and such that $K\defeq \overline{D}\subset G$. Then $f^{-1}(D)$ is connected by 
     Proposition~\ref{prop:smaller}, as is $f^{-1}(K) = \overline{f^{-1}(D)}$. It follows from Proposition~\ref{prop:compactconnected} that
    $S(f)\subset K\subset G$, as required.
\end{proof}
 
We next record a strengthening of Theorem~\ref{thm:itdoeshold}~\ref{infaccessible} 
   that concerns the accessibility of \emph{direct asymptotic values} on
   $\partial G_1$ from $G_1$. 
   Here $a\in\Ch$ is a direct asymptotic value of $f$ if there is an open connected neighbourhood $\Delta$ of $a$ in $\Ch$ and a connected component $\tilde{\Delta}$ of
  $f^{-1}(\Delta)$ such that $a\notin f(\tilde{\Delta})$. Observe that every Picard exceptional value (i.e., a value $a$ for which $f^{-1}(a)$ is finite) is direct; 
   in particular, $\infty$ is always a direct asymptotic value of~$f$.
\begin{proposition}[Accessible asymptotic values on boundaries]
\label{prop:Picard}
Let $f$ be a transcendental entire function and let $\xi\in\Ch$ be a direct asymptotic value of $f$. Suppose that 
$G_1, G_2$ are disjoint simply-connected domains such that $f^{-1}(G_1)$ is connected. If $\xi$ is an accessible boundary point of $G_1$, then
$f^{-1}(G_2)$ is disconnected.
\end{proposition}
\begin{proof}
%
Let $\Delta$ and $\tilde{\Delta}$ be as in the definition of a direct asymptotic value. 
If $\Delta$ is chosen sufficiently small, then 
   every value of $\Delta\setminus\{a\}$ has infinitely many preimages in $\Delta$
     (this follows once more from \cite[Theorem~4']{heins}). 
 By assumption, $G_1\cap \Delta$ has a connected component $\Delta_1$ from which $a$ is accessible. 
 Choose a point $z_1\in \Delta_1$ which is not a critical value. Then $z_1$ has  infinitely many simple preimages in $\tilde{\Delta}$, 
  each corresponding to a different branch of $f^{-1}$. Using Theorem~\ref{theo:gross}, we construct a polygonal arc $\gamma\subset \Delta_1$ from $z_1$ to $\xi$, possibly with infinitely many pieces, such that these inverse branches can all be continued along $\gamma\setminus\{\xi\}$.

Let $p \in \tilde{\Delta}\cap f^{-1}(z_1)$, and let $\gamma_p\subset \tilde{\Delta}$ be the component 
  of $f^{-1}(\gamma)$ containing $p$. Since $\xi$ has no preimages in $U$, the arc $\gamma_p$ 
  connects $p$ to $\infty$. Pick a second preimage $q\neq p$ of $z_1$ in $\tilde{\Delta}$, 
  and define $\gamma_q$ analogously. We also join $p$ and $q$ by an arc $\tau$ in $f^{-1}(G_1)$ that does not intersect $\gamma_p \cup \gamma_q$. 
    Note that $\gamma_p \cup \tau \cup \gamma_q\subset f^{-1}(G_1)$ separates the 
    plane into two complementary
    components; we claim that each such component $D$ 
    intersects $f^{-1}(G_2)$. This 
    implies that $f^{-1}(G_2)$ is disconnected, as required.

Let $B$ denote the union of $f(\tau)$ and its bounded complementary components. Then $B$ is a compact subset of
  $G_1$. Let $w \in \gamma$ be a point so close to $\xi$ that 
  $B$ does not intersect the arc of $\gamma$ connecting $w$ to $\xi$. 
   Then $w$ has a preimage on $\gamma_p$; say $\tilde{w}$. Let $W\subset G_1$ be a small neighbourhood of $w$ disjoint from $B$, and let $\tilde{W}$ be the component of the preimage of $W$ containing $\tilde{w}$. We may choose $\zeta \in D \cap \tilde{W}$ such that $\zeta$ is not a critical point and $f(\zeta) \notin \gamma$.
   By choice of $w$, and since $G_1$ is simply-connected, the set 
   $B\cup\gamma$ does 
   not separate $f(\zeta)$ from $G_2$.
      
Using the Gross star theorem~\ref{theo:gross}, we find a polygonal curve 
  $\Gamma\subset \C\setminus (B\cup\gamma)$
     from $f(\zeta)$ to a point of $G_2$ in such a way that the inverse branch of $f$ that sends $f(\zeta)$ to $\zeta$ can be continued along $\Gamma$. 
     Then the preimage of $\Gamma$ is a curve in $D$ that ends at a point of 
      $f^{-1}(G_2)$. This completes the proof.
\end{proof}

\begin{proof}[Proof of Theorem~\ref{thm:itdoeshold}]
Let $f$ be a transcendental entire function, and let $G_1, G_2$ be disjoint simply-connected domains such that $f^{-1}(G_1)$ is connected. 
  In turn, we shall conclude from each assumption in Theorem~\ref{thm:itdoeshold} that $f^{-1}(G_2)$ is disconnected.

We begin with~\ref{boundeddomain}, so suppose that $G_1$ 
   is bounded and $\overline{G_1}$ does not separate $G_2$ from infinity. Set $K\defeq \overline{G_1}$
  and let $W$ be the unbounded connected component of $\C\setminus K$.  Then, by assumption, $f^{-1}(K)=\overline{f^{-1}(G_1)}$ is connected
  and $G_2\subset W$.    
   By Proposition~\ref{prop:compactconnected}, $G_2 \cap S(f)=\emptyset$. Hence $f$ is univalent on every component of $f^{-1}(G_2)$, and in particular
  $f^{-1}(G_2)$ is disconnected.

Next, we turn to~\ref{compactlyc}.
If $G_1\cap S(f)$ 
 is compactly contained in $G_1$ and $f^{-1}(G_1)$ is connected, then by Corollary~\ref{cor:compactlyc},
    $S(f)\subset G_1$. It follows as above that $f^{-1}(G_2)$ is disconnected.
    Clearly~\ref{inS} is a direct consequence of~\ref{compactlyc}.

Now suppose that \ref{onlyonesingularity} holds. We use a technique similar to the proof of part (iii) of the theorem of \cite{walteralex2}. By assumption,
   $f^{-1}(G_1)$ contains a Jordan curve $\Gamma$, unbounded in both directions, such that $f$ tends to an asymptotic value $a_1\in \Ch$ in one direction,
   and to a (possibly different) asymptotic value $a_2\in\Ch$ in the other. Since both ends of $\Gamma$ represent different 
   transcendental singularities, there is $\eps>0$ such that 
  $\Gamma$ has unbounded intersection with two different components of $f^{-1}(\{ z \in \C \colon \dist^{\#}(z,\{a_1,a_2\}) < \epsilon \})$.
   (Here $\dist^{\#}$ denotes spherical distance.)

The Jordan curve $\Gamma$ divides the plane into two components. As 
 in the proof of Proposition~\ref{prop:Picard}, we claim that each such component $D$ must meet $f^{-1}(G_2)$, which is therefore disconnected.
Otherwise, choose a point $w \in G_2$, and consider the function \[g(z) = \frac{1}{f(z) - w}.\] Then $g$ is bounded and holomorphic in $D$ and has finite limits as $z$ tends to infinity along~$\Gamma$. By a theorem of Lindel\"of \cite[p.~9, \P~5]{lindelof15}, we have $a_1=a_2$ and $f(z)\to a_1$ as $z\to\infty$ in $D$. 
  (See also \P\P~36, 39 and 61 in \cite[Chapter~III]{nevanlinnagerman}.) This is a contradiction, as 
 $D \cap f^{-1} (\{ z \in \C \colon \dist^{\#}(z,\{a_1,a_2\}) = \epsilon \})$ has an unbounded connected component.

We now turn to~\ref{sharedboundary}. By way of contradiction, suppose that $f^{-1}(G_2)$ is connected and $\zeta$ is accessible from both $G_1$ and $G_2$. We can assume, by Proposition~\ref{prop:Picard}, that $\zeta$ is not a Picard exceptional point. 
Let $\zeta_p$ and $\zeta_q$ be two distinct 
  preimages of $\zeta$, and let $\Delta$ be a disc around $\zeta$. If $\Delta$ is sufficiently small, then $\partial \Delta$ meets both $G_1$ and $G_2$, and for each $j \in \{p, q\}$ the component $\Delta_j$ of $f^{-1}(\Delta)$ 
  that contains $\zeta_j$ has all the following properties;
\begin{itemize}
\item $\Delta_j$ is bounded and does not meet the other preimage component;
\item $\overline{\Delta_j}$ is mapped to $\overline{\Delta}$ as a degree $d$ branched covering, 
    where $d$ is the local degree of $f$ at $\zeta_j$; 
\item in particular, $\Delta_j$ 
    contains no critical points of $f$ except possibly $\zeta_j$, and also $\Delta_j \cap f^{-1}(\zeta) = \{\zeta_j\}$.
\end{itemize}

For $j\in \{1, 2\}$, let $\gamma_j$ be an arc connecting $\zeta$ to a point of $G_j\cap \partial \Delta$ with the property that 
    $\interior(\gamma_j)\subset \Delta\cap G_j$. 
 Set $\gamma = \gamma_1\cup \gamma_2$. Let $z_1$ and $z_2$ be the two endpoints of $\gamma$, with $z_j\in G_j\cap \partial \Delta$. 

For each $j \in \{p, q\}$, let $\tilde{\gamma}_j$ be the component of $f^{-1}(\gamma)$ containing $\zeta_j$. Then  
  $\tilde{\gamma}_j$ is either an arc (corresponding to the case 
  where $\zeta_j$ is not a critical point), or a tree with just one vertex of order greater than one (at $\zeta_j$). 
  In either case, there is an arc $\gamma_j\subset \tilde{\gamma}_j$ 
   such that $f\colon \gamma_j\to\gamma$ is a homeomorphism. We let $p_k$ ($k\in\{1,2\}$) be the point on 
   $\gamma_p$ whose image is $z_k$. We similarly let $q_k$ be the point on $\gamma_q$ whose image 
   is $z_k$.

For $k \in \{1, 2\}$, we can join $p_k$ and  $q_k$ by an arc 
  $\beta_k \subset f^{-1}(G_k)$. Set $C_k\defeq f(\beta_k)$
   and consider the set $A_k\defeq C_k \cup \gamma_k \subset G_k\cup \{\zeta\}$. 
       Since $G_k$ is simply-connected, every bounded connected component
       of $\C\setminus A_k$ is contained in $G_k$, and does not 
       contain $\zeta$ on its boundary. Furthermore, $A_1\cap A_2=\{\zeta\}$ is connected, 
       so by Janiszewski's theorem there are no other bounded connected 
       complementary components of
       $A_1\cup A_2 = C_1\cup C_2\cup \gamma$. This is impossible  by 
        Proposition~\ref{prop:baker}. We have obtained the desired contradiction.

Observe that \ref{infaccessible} is an immediate consequence of Proposition~\ref{prop:Picard}.

Now assume that \ref{disjointclosure} holds; i.e.\ that $\overline{G_1} \cap \overline{G_2} = \emptyset$, and suppose by way of contradiction that $f^{-1}(G_2)$ is connected. 
It follows by \cite[Theorem~VI.3.1]{whyburn} that there is a Jordan curve $\Gamma$, 
   which may pass through $\infty$, 
 that separates the plane into two components, one containing $\overline{G_1}$ and the other containing $\overline{G_2}$. 

For $j\in\{1,2\}$, let $V_j$ be the connected component of $\C\setminus\Gamma$ containing $G_j$. 
   Then $f^{-1}(V_j)$ is connected by Lemma~\ref{lem:increasingsets}. 
   If $\Gamma$ is bounded, then one of the $V_j$, say $V_1$,  is bounded. Clearly 
    $\overline{V_1} = V_1\cup\Gamma$ does not separate $V_2$ from
    infinity, 
    contradicting~\ref{boundeddomain}.

  On  the other hand, if $\Gamma$ is unbounded, then both 
    $V_j$ are simply-connected
  and every point of $\partial V_1 = \partial V_2 = \Gamma$ is accessible from both
   $V_1$ and $V_2$. This contradicts~\ref{sharedboundary}.

Finally, as noted earlier, \ref{walteralexresult} is  an immediate consequence of \cite[Theorem~1]{walteralex}.
\end{proof}

%
%
%
%

\section{Appendix}
\label{sect:appendix}
In this final part of the paper we list, for the convenience of the research community, various papers whose results are impacted by the problem found. Note that this list is not necessarily complete; some of the papers in this list are very heavily cited, as they include other influential results that are not in doubt, and so it is difficult to be sure that every issue has been identified.

Firstly, we have identified six papers that have a common flaw in a proof, and so certain results in these papers must be considered \emph{open}. These are as follows:
\begin{enumerate}[(i)]
\item The result of \cite{baker1970} (this paper contains only one result).
\item The proof of \cite[Theorem 2]{bakernormality}. Note that the other results of \cite{bakernormality} are not affected by the problem.
\item The proof of \cite[Lemma 11]{eremenkoandlyubich}. No other results of \cite{eremenkoandlyubich} are affected by the problem. 
\item The proof of \cite[Theorem K]{Dominguez}. No other results of \cite{Dominguez} are affected by the problem.
\item The proof of \cite[Lemma 2.2]{ling}, and hence \cite[Theorem 2.1]{ling}, which depends on this lemma.
\item The proofs of \cite[Theorem 3 and Theorem 4]{MR2011920}.
\end{enumerate}

Most papers that use these theorems are dynamical, and so our constructions give no additional information on whether their results are correct. We have identified just one explicitly stated
  result that is now seen to be false. This is \cite[Theorem 2]{walteralex}, which states that if $f$ is an entire function of finite order, and $a \in\C$ is either a critical value or a locally omitted value, then any simply-connected region that does not contain $a$ has disconnected preimage. Theorem~\ref{thm:easyexample} can be seen to be a counter-example to this assertion. No other papers use \cite[Theorem 2]{walteralex}.

We mention also that the preprint \cite{domingueznew} attempts to give
  a ``complementary proof'' of Baker's original result regarding completely invariant Fatou
  components. Unfortunately, the proof does not use the dynamics of 
  the function in an essential way. Therefore it also runs afoul of our
  construction, and would effectively contradict Theorem~\ref{thm:easyexample}. 

We end with three lists of results that use the potentially flawed theorems. The first list contains those results 
that use the potentially flawed theorems, but can been seen nonetheless to hold.
\begin{enumerate}[(i)]
\item The proof of \cite[Theorem H]{MR1785150} uses \cite[Lemma 11]{eremenkoandlyubich}. This result applies only to the class $\mathcal{S}$, and so in fact the result can be recovered by using Theorem~\ref{thm:itdoeshold}\ref{inS} instead.
\item The proof of \cite[Theorem B]{MR1836425} uses \cite{baker1970}. This result can also be recovered by using Theorem~\ref{thm:itdoeshold}\ref{inS} instead.
\item The proof of \cite[Theorem 1.2]{MR2801622} uses the result of \cite{baker1970}. However, this dependence can easily be shown not to be essential (since a slight adaptation of the proof gives the result for $f^p$ in the case that there is a completely invariant p-cycle of Fatou components), and so \cite[Theorem~1.2]{MR2801622} remains valid.
\item The proof of \cite[Lemma 3.3]{MR3082543} uses the result of \cite{baker1970}, and this lemma is then used to prove \cite[Theorem 3.1]{MR3082543}. However, it is easy to see how to weaken the statements of \cite[Lemma 3.3 and Theorem 3.1]{MR3082543} in such a way that \cite{baker1970} is no longer needed, but the proof of \cite[Theorem 1.1]{MR3082543}, which is one of the main results of \cite{MR3082543}, still holds.
\item The proof of \cite[Proposition 2.6]{MR3424895} requires \cite[Lemma 11]{eremenkoandlyubich}. Since the functions in \cite{MR3424895} are in the class $\mathcal{S}$, this can be recovered by using using Theorem~\ref{thm:itdoeshold}\ref{inS} instead.
\end{enumerate}
The second list contains those that need to be considered as still open.
\begin{enumerate}[(i)]
\item Bhattacharyya \cite{MR719610}, Hinkkanen \cite{MR1279126} and Fang \cite{MR1604503} each used \cite{baker1970} to prove the analogous result for analytic self-maps of the punctured plane. These results are not cited by any other paper.
\item The result of \cite{baker1970} is used to prove \cite[Lemma 3.1]{MR2011920}, which in turn is used to prove \cite[Theorem 1]{MR2011920}. This paper also uses \cite[Lemma 11]{eremenkoandlyubich} to prove \cite[Lemma 4.5]{MR2011920} and thence \cite[Theorem 2]{MR2011920}. These results are also quoted in \cite{ling}. Note that \cite[Theorem 1]{MR2011920} was generalised to a wider class of functions in \cite[Lemma 3.3]{MR2238646}, and this result does not require any of the open theorems.
\end{enumerate}
Finally we note two implications for survey papers: 
\begin{enumerate}[(i)]
\item The result of \cite{baker1970} is stated as \cite[Theorem 17]{walteriteration}; however, there are no significant implications elsewhere in this paper.
\item The result \cite[Lemma 11]{eremenkoandlyubich} is stated and proved as \cite[Theorem 4.7]{MR1015124}; however, there are no  implications elsewhere in this paper.
\end{enumerate}
%
%
%
%
%

\providecommand{\bysame}{\leavevmode\hbox to3em{\hrulefill}\thinspace}

\providecommand{\href}[2]{#2}


\begin{thebibliography}{BFJK17}

\bibitem[Bak70]{baker1970}
I.~N. Baker, \emph{Completely invariant domains of entire functions},
  Mathematical {E}ssays {D}edicated to {A}. {J}. {M}acintyre, Ohio Univ. Press,
  Athens, Ohio, 1970, pp.~33--35.

\bibitem[Bak75]{bakernormality}
\bysame, \emph{The domains of normality of an entire function}, Ann. Acad. Sci.
  Fenn. Ser. A I Math. \textbf{1} (1975), no.~2, 277--283.

\bibitem[BD00]{MR1785150}
I.~N. Baker and P.~Dom{\'{\i}}nguez, \emph{Some connectedness properties of
  {J}ulia sets}, Complex Variables Theory Appl. \textbf{41} (2000), no.~4,
  371--389.

\bibitem[BDH01]{MR1836425}
I.~N. Baker, P.~Dom{\'{\i}}nguez, and M.~E. Herring, \emph{Dynamics of
  functions meromorphic outside a small set}, Ergodic Theory Dynam. Systems
  \textbf{21} (2001), no.~3, 647--672.

\bibitem[BE95]{bergweilereremenkosingularities}
 W.~Bergweiler and A.~Eremenko, \emph{On the singularities of the inverse to a meromorphic function of finite order},
Rev. Mat. Iberoamericana \textbf{11} (1995), no. 2, 355--373. 

\bibitem[BE08a]{walteralex}
W.~Bergweiler and A.~E. Eremenko, \emph{Direct singularities and completely
  invariant domains of entire functions}, Illinois J. Math. \textbf{52} (2008),
  no.~1, 243--259.

\bibitem[BE08b]{walteralex2}
Walter Bergweiler and Alexandre Eremenko, \emph{Meromorphic functions with two
  completely invariant domains}, Transcendental dynamics and complex analysis,
  London Math. Soc. Lecture Note Ser., vol. 348, Cambridge Univ. Press,
  Cambridge, 2008, pp.~74--89.

\bibitem[Ber93]{walteriteration}
W.~Bergweiler, \emph{Iteration of meromorphic functions}, Bull. Amer. Math.
  Soc. (N.S.) \textbf{29} (1993), no.~2, 151--188.

\bibitem[Ber02]{walter}
\bysame, \emph{A question of {E}remenko and {L}yubich concerning completely
  invariant domains and indirect singularities}, Proc. Amer. Math. Soc.
  \textbf{130} (2002), no.~11, 3231--3236.

\bibitem[BFJK17]{MR3581221}
Krzysztof Bara\'nski, N\'uria Fagella, Xavier Jarque, and Bogus{\l}awa
  Karpi\'nska, \emph{Accesses to infinity from {F}atou components}, Trans.
  Amer. Math. Soc. \textbf{369} (2017), no.~3, 1835--1867.

\bibitem[BFR15]{walternurialasse}
W.~Bergweiler, N.~Fagella, and L.~Rempe{-Gillen}, \emph{Hyperbolic entire
  functions with bounded {F}atou components}, Comment. Math. Helv. \textbf{90}
  (2015), no.~4, 799--829.

\bibitem[Bha83]{MR719610}
P.~Bhattacharyya, \emph{On completely invariant domains and domains of
  normality}, J. Math. Phys. Sci. \textbf{17} (1983), no.~4, 427--434.

\bibitem[CW03]{MR2011920}
C-L. Cao and Y-F. Wang, \emph{On completely invariant {F}atou components}, Ark.
  Mat. \textbf{41} (2003), no.~2, 253--265.

\bibitem[Den15]{MR3424895}
A.~Deniz, \emph{Convergence of rays with rational argument in hyperbolic
  components---an illustration in transcendental dynamics}, Nonlinearity
  \textbf{28} (2015), no.~11, 3845--3871.

\bibitem[Dom98]{Dominguez}
P.~Dom\'inguez, \emph{Dynamics of transcendental meromorphic functions}, Ann.
  Acad. Sci. Fenn. Math. \textbf{23} (1998), no.~1, 225--250.

\bibitem[DS18]{domingueznew}
P.~Dom{\'{\i}}nguez and G.~{Sienra}, \emph{A complementary proof of {B}aker's
  theorem of completely invariant components for transcendental entire
  functions}, Preprint, arXiv:1803.04598v1 (2018).

\bibitem[EL89]{MR1015124}
A.~E. Eremenko and M.~Yu. Lyubich, \emph{The dynamics of analytic
  transformations}, Algebra i Analiz \textbf{1} (1989), no.~3, 1--70.

\bibitem[EL92]{eremenkoandlyubich}
\bysame, \emph{Dynamical properties of some classes of entire functions}, Ann.
  Inst. Fourier (Grenoble) \textbf{42} (1992), no.~4, 989--1020.

\bibitem[Ere13]{eremenkotalk}
A.~E. Eremenko, \emph{Singularities of inverse functions}, Presented at the
  {ICMS} conference ``The role of complex analysis in complex dynamics'', 2013.

\bibitem[Fan97]{MR1604503}
L.~Fang, \emph{Completely invariant domains of holomorphic self-maps on
  {$\mathbb{C}^*$}}, J. Beijing Inst. Tech. \textbf{6} (1997), no.~3, 187--191.

\bibitem[{Hei}57]{heins}
Maurice {Heins}, \emph{{Asymptotic spots of entire and meromorphic
  functions.}}, {Ann. Math. (2)} \textbf{66} (1957), 430--439 (English).

\bibitem[Hin94]{MR1279126}
A.~Hinkkanen, \emph{Completely invariant components in the punctured plane},
  New Zealand J. Math. \textbf{23} (1994), no.~1, 65--69.

\bibitem[Lin08]{ling}
Q.~Ling, \emph{Some dynamical properties of transcendental meromorphic
  functions}, J. Math. Anal. Appl. \textbf{340} (2008), no.~2, 954--958.

\bibitem[Lin15]{lindelof15}
E. {Lindel\"of}, \emph{Sur un principe g\'en\'eral de l'analyse et ses applications \`a la th\'eorie de la repr\'esentation conforme},
  {Acta Soc. Fennicae \textbf{46},  no.~4 (1915), 1--35.}

\bibitem[Nev53]{nevanlinnagerman}
Rolf Nevanlinna, \emph{Eindeutige analytische {F}unktionen}, Die Grundlehren
  der mathematischen Wissenschaften in Einzeldarstellungen mit besonderer
  Ber\"ucksichtigung der Anwendungsgebiete, Bd XLVI, Springer-Verlag, Berlin,
  1953, 2te Aufl.

\bibitem[NZC06]{MR2238646}
T.~W. Ng, J-H. Zheng, and Y~Y. Choi, \emph{Residual {J}ulia sets of meromorphic
  functions}, Math. Proc. Cambridge Philos. Soc. \textbf{141} (2006), no.~1,
  113--126.

\bibitem[Osb13]{MR3082543}
J.~W. Osborne, \emph{Spiders' webs and locally connected {J}ulia sets of
  transcendental entire functions}, Ergodic Theory Dynam. Systems \textbf{33}
  (2013), no.~4, 1146--1161.

\bibitem[RS11]{MR2801622}
P.~J. Rippon and G.~M. Stallard, \emph{Boundaries of escaping {F}atou
  components}, Proc. Amer. Math. Soc. \textbf{139} (2011), no.~8, 2807--2820.

\bibitem[Why42]{whyburn}
G.~T. Whyburn, \emph{Analytic {T}opology}, American Mathematical Society
  Colloquium Publications, v. 28, American Mathematical Society, New York,
  1942.

\end{thebibliography}
\end{document}